\documentclass[11pt,eqno]{amsart}

\usepackage{bbm}
\usepackage{mathrsfs}
\usepackage{enumerate}
\usepackage{amssymb}
\usepackage{amsmath}
\usepackage{amsthm}
\usepackage{amsfonts}
\usepackage{amsmath,amssymb,txfonts}
\usepackage{amssymb}
\usepackage{amsxtra}
\usepackage{amsmath}
\usepackage{txfonts}
\usepackage[mathscr]{eucal}
 \usepackage{color}

\usepackage[cp1252]{inputenc}

\usepackage{mathrsfs}
\usepackage{graphicx}

\allowdisplaybreaks

\def\rr{{\mathbb R}}
\def\rn{{{\rr}^n}}

\def\fz{\infty}

\def\dist{{\mathop\mathrm{\,dist\,}}}
\def\loc{{\mathop\mathrm{\,loc\,}}}

\def\lip{{\mathop\mathrm{\,Lip}}}

\def\dz{\delta}
\def\bdz{\Delta}
\def\ez{\varepsilon}

\def\bz{\beta}

\def\gz{{\gamma}}

\def\sz{\sigma}

\def\wz{\widetilde}

\def\Xint#1{\mathchoice	{\XXint\displaystyle\textstyle{#1}}	{\XXint\textstyle\scriptstyle{#1}}	{\XXint\scriptstyle\scriptscriptstyle{#1}}	{\XXint\scriptscriptstyle\scriptscriptstyle{#1}}	\!\int}\def\XXint#1#2#3{{\setbox0=\hbox{$#1{#2#3}{\int}$}	\vcenter{\hbox{$#2#3$}}\kern-.5\wd0}}\def\dashint{\Xint-}

\newtheorem{thm}{Theorem}[section]
\newtheorem{lem}{Lemma}[section]

\newtheorem{rem}{Remark}[section]

\newtheorem{defn}{Definition}[section]

\numberwithin{equation}{section}

\begin{document}

\arraycolsep=1pt

\title[A quantitative second order Sobolev regularity]{A quantitative second order Sobolev regularity\\ for (inhmogeneous) normalized $p(\cdot)$-Laplace equations 
}
{ \footnotetext{\hspace{-0.35cm}
\endgraf
 \endgraf {\it Key words and phrases:}  
 normalized $p(x)$-Laplacian,  strong $p(x)$-Laplacian,
 second order regularity, quasiregular mapping
 \endgraf This project was  supported by  NSFC (No.   12025102) and by the Fundamental Research Funds for the Central Universities.
}

\author{Yuqing Wang and Yuan Zhou}


\date{\today}

\maketitle

\begin{center}
\begin{minipage}{13.5cm}\small{\noindent{\bf Abstract}\quad
Let $\Omega$ be a domain of $\mathbb R^n$ with $n\ge 2$ and $p(\cdot)$  be  a local Lipschitz funcion in $\Omega$ with   $1<p(x)<\infty$ in $\Omega$. 
We build up an interior {\color{red}quantitative} second order Sobolev regularity for  
   the normalized $p(\cdot)$-Laplace equation $-\Delta^N_{p(\cdot)}u=0$  in $\Omega$  as well    as  the corresponding inhomogeneous equation {\color{red}$-\Delta^N_{p(\cdot)}u=f$} in $\Omega$  with $f\in C^0(\Omega)$. 
  In particular,   given 
  any   viscosity solution $u$   to $-\Delta^N_{p(\cdot)}u=0$  in $\Omega$,      we prove the following: 
  \begin{enumerate}
\item[(i)] in dimension $n=2$, for any subdomain $U\Subset\Omega$ and any $\bz\ge 0$,  one has 
$|Du|^\bz Du\in L^{2+\dz}_\loc(U)$ with a {\color{red}quantitative} upper bound, and moreover,    the map $(x_1,x_2)\to |Du|^\bz(u_{x_1},-u_{x_2})$ is quasiregular  in $U$ in the sense     that  
  $$|D[|Du|^\bz Du]|^2\le -C\det D[|Du|^\bz Du] \quad \mbox{a.\,e. in $U$}.$$  
  
  \item[(ii)] in dimension $n\ge3$,  for any subdomain $U\Subset\Omega$ with 
   $ \inf_U p(x)>1$ and $\sup_Up(x)<3+\frac2{n-2}$,  one has 
$D^2u\in L^{2+\dz}_\loc(U)$  with a {\color{red}quantitative} upper bound,   and also with a pointwise upper bound 
$$|D^2u|^2\le -C\sum_{1\le i<j\le n}[u_{x_ix_j}u_{x_jx_i}-u_{x_ix_i}u_{x_jx_j}] \quad \mbox{a.\,e. in $U$}.$$
   \end{enumerate}
    Here constants $\dz>0$ and $C\ge 1$  are independent of $u$. 
  These extend the  related results obtaind by  Adamowicz-H\"ast\"o \cite{AH2010} when $n=2$ and $\bz=0$. } 
\end{minipage}
\end{center}

\section{Introduction}

Let $\Omega$ be a  domain (open connected set) of $\rn$ with $n\ge2$. Suppose that 
   $p:\Omega\to\rr$ is a continuous function   with $1<p(x)<\fz$ for all $x\in \Omega$; for short,  $p\in C^0(\Omega;(1,\fz))$.   
In general,  the function $p$ is not necessary to be a constant; to  emphasis  this, 
we   write  $p(\cdot)$  or  $p(x)$.  Denote by 
 $-\Delta^N_{p(\cdot)}$    the  normalized $p(\cdot)$-Laplacian, that is,   
 $$\Delta^N_{p(\cdot)}u:=\Delta u+(p(\cdot)-2)\frac{\Delta_\fz u}{|Du|^2},$$ 
 where $\Delta_\fz$ is the infinity Laplacian given by
 $\Delta_\fz u:=D^2uDu\cdot Du.$  
 Observe  that $\Delta^N_{p(\cdot)}u$ is not of divergence form.
  Consider the normalized $p(\cdot)$-Laplace equation
\begin{equation}\label{eq1}
-\Delta^N_{p(\cdot)}u =0\quad \mbox{ in $ \Omega$,}
\end{equation}
 and also, for any $f \in C^0(\Omega)$, 
 the corresponding inhomogeneous   equation  
\begin{equation}\label{eq1f}-\Delta^N_{p(\cdot)}u=f\quad  \mbox{ in $ \Omega$.}\end{equation}
   Viscosity solutions  to \eqref{eq1} and \eqref{eq1f}
  are given in  
  Definition \ref{def2-1}.

 This paper aims to build up  a {\color{red}quantitative} second order Sobolev regularity 
 for viscosity solutions to   the equations \eqref{eq1} and   \eqref{eq1f} under an additional assumption 
 $p(\cdot)\in C^{0,1}(\Omega)$, that is, $p(\cdot)$ is  local Lipschitz  in $\Omega$ 
in the sense that 
\begin{equation}\label{loclip}\lip(p(\cdot),U)=\sup_{x,y\in U}\frac{|p(x)-p(y)|}{|x-y|}<\fz \ \mbox{for any subdomain $U\Subset\Omega$. }\end{equation}
Throughout the whole paper, $U\Subset\Omega$ means that $U$ is a bounded  domain and its closure $\overline U\subset\Omega$. 
For convenience, we  write  $$
\mbox{$p(\cdot)\in C^{0,1}(\Omega;(1,\fz))$ whenever 
$p(\cdot)\in C^{0,1}(\Omega)$ and $1<p(x)<\fz$ for all  $x\in \Omega$. }$$
We also use $p_\pm^E$ to denote the supremum/infimum  of $p(\cdot)$ in the set $E \subset\Omega$, 
that is,
$$ p^E_+=\sup\{p(x)|x\in E\}  \ \mbox{and}\ p^E_-=\min\{p(x)|x\in E\}.$$
Obviously,  $1<p^U_-\le p^U_+<\fz$ for any domain $U\Subset\Omega$. 
Moreover, we use $C(a,b,\cdots)$ to denote a positive constant which depending only on the parameters $a,b,\cdots$, but whose value may   vary from line to line.   
 
 Regards to  the normalized $p(\cdot)$-Laplace equation    \eqref{eq1}, our main results are stated  in Theorem  1.1 when
  the dimension $n=2$ and in Theorem  1.2 when the dimension $n\ge3$.

 \begin{thm} Let $\Omega\subset\rr^2$ be a   domain, $p(\cdot)\in C^{0,1}(\Omega;(1,\fz))$ and   $\bz\in[0,\fz)$. 
For any   viscosity solution $u$   to the equation \eqref{eq1}, the following hold.  
 \begin{enumerate}
 \item[(i)]  One has  $  |Du | ^{\beta }Du \in W^{1,2}_\loc(\Omega)$, and in particular, $D^2u\in L^2_\loc(\Omega)$.  Moreover, 
 the map $(x_1,x_2)\to |Du|^\bz ( u_{x_1},-  u_{x_2} )$  is locally quasiregular in $\Omega$, that is, 
 \begin{align}\label{bound3} \mbox{ for any domain $U\Subset\Omega$, one has  }  
|D[|Du|^{ \bz}Du] |^2\le - C(p^U_\pm,\beta)\det D[|Du|^{ \bz}Du] \quad {\rm a.e.\  in} \    U.
\end{align}
 
 \item[(ii)]  Given any    subdomain $U\Subset\Omega$,   
  one has 
 $|Du|^{ \bz}Du\in W_\loc^{1,2+\dz}(U)$ for some $\dz=\dz( p^U_\pm,\bz)>0$ with a {\color{red}quantitative} upper bound 
\begin{align}\label{bound2}
 \left(\dashint_{\frac 14B}|D(|Du|^\beta Du)|^{2+\dz}\,dx\right)^{\frac 1{ 2+\dz }} 
 &\le C(p^U_\pm, \beta )  \frac1{R } \left(\dashint_{ {\frac34} B }   ||Du |^{\beta } Du -\vec c |^2 \,dx\right)^{1/2}  \quad \forall \  B=B(z,R)\Subset U.
\end{align}
%

   \end{enumerate}
  \end{thm}
%
%

To state our main results for the {\color{red}equation} \eqref{eq1} in dimension $n\ge3$,   
we need some extra but necessary notation and notions concerning  $p(\cdot)$ and $n$. 
Firstly, we work with $\sigma_2(D[|Du|^\bz Du])$ as a natural {\color{red}substitute} 
of    the determinant $-\det D[|Du|^{ \bz}Du]$ as considered  in   Theorem 1.1 (i) in dimension $n=2$, {\color{red} where  $\sigma_2(D[|Du|^\bz Du])=-\det D[|Du|^\bz Du]$.
We  refer to \cite{dpzz19} for some reasons to use $\sigma_2$ in higher dimensions.} 
Recall  that, 
 $$\sz_2(M ):={-}\sum_{1\le i<j\le n}[m_{ii}m_{jj}-m_{ij}m_{ji}]\quad 
 \mbox{ for any $n\times n$ matrix $M =(m_{ij})_{1\le i,j\le n}$}.$$
 Obviously, when $n=2$, $\sz_2(M )=-\det M$. 
Next,   for any   $t>1$, set
 \begin{align}\label{betastar}\bz_\star(n,t):=-1+\frac{n-2}{{2}(n-1)}(t-1).  
 \end{align} 
Note that $\bz_\star(2,t)\equiv-1$ in dimension $n=2$. 
In dimension $n\ge3$,  $\bz_\star(n,t)$ is strictly increasing in $t$, 
$$\mbox{$\bz_\star(n,t)>-1$ for all $t>1$,  and $\bz_\star(n,t)=0$   if and only if  }\ t= 3+\frac{2}{n-2}. $$
  Moreover,  write  
 $$\Omega^\ast_n:=\left\{x\in \Omega|p(x)<3+\frac 2{n-2}\right\}.$$
Obviously, $\Omega^\ast_2=\Omega$ in dimension $n=2$.

\begin{thm}Let $\Omega\subset\rn$ be a   domain with $n\ge3$  and $p(\cdot)\in C^{0,1}(\Omega;(1,\fz))$. 
For any  viscosity solution  $u$ to the equation \eqref{eq1}, the following hold.

\begin{enumerate}
\item[(i)]  One has $D^2u\in L^2_\loc(\Omega^\ast_n)$. Given any    subdomain $U\Subset\Omega_n^\ast$, one has 
$D^2u\in L^{2+\dz}_\loc(U)$ for some $ \dz=\dz(p^U_\pm,n)>0$, and moreover, 
a pointwise upper bound  
 \begin{align}\label{eq4.x1}
 |D^2u |^2\le C 
 (p^U_\pm,n )\sigma_2(D^2u) \quad \mbox{a.\,e. in $ U$}. 
\end{align}
\item[(ii)]  

Given any    subdomain $U\Subset\Omega$ and any $\bz\in(\bz_\star(n,p^U_+) ,\fz)\cap[0,\fz)$,
one has  $ |Du|^{ \bz}Du \in {W}^{1,{2+\dz}}_\loc (U)$  for some $ \dz=\dz(p^U_\pm,n,\bz)>0$,   a {\color{red}quantitative} upper bound  \eqref{bound2} with 
 the constant $C( p^U_\pm,\bz) $ therein replaced by $C( p^U_\pm,n,\bz)$,
%
 and moreover, a pointwise upper bound 
 $$
 |D[|Du|^\bz Du]|^2\le C 
 (p^U_\pm,n,\beta)\sigma_2(D[ |Du| ^{{\bz }}Du]) \quad \mbox{a.\,e. in $U$}. $$
 
  
\end{enumerate}
\end{thm}

Our main results to the inhomogeneous  normalized $p(\cdot)$-Laplace equation  \eqref{eq1f} are given in Theorem  1.3 when
    $n=2$ and in Theorem  1.4 when $n\ge3$. 
\begin{thm} 
Let $\Omega\subset\rr^2$ be a   domain,  $p(\cdot)\in C^{0,1}(\Omega;(1,\fz))$ and $f\in C^0(\Omega)$.  Let $\bz\in[0,\fz)$. 
For   any  viscosity solution $u$ to the equation \eqref{eq1f}, the following hold. 
\begin{enumerate}
\item[(i)] One has   $D[ |Du | ^{\beta }Du ]\in L^2_\loc(\Omega)$ and in particular, $D^2u\in L^2_\loc(\Omega)$. 
\item[(ii)] Given   any subdomain $U\Subset\Omega$,   one has 
  $|Du|^{ \bz}Du\in W_\loc^{1,2+\dz}(U )$ for some  $\delta=\dz(p^U_\pm,\bz)>0$ with  a  {\color{red}quantitative} upper bound
    \begin{align}\label{bound4}
&\left(\dashint_{\frac 14B}|D(|Du|^\beta Du)|^{2+\dz}\,dx\right)^{\frac 1{ 2+\dz }} \nonumber\\
 &\le C(p^U_\pm, \beta ) \left[\frac1{R } \left(\dashint_{ {\frac34} B }   |Du |^{2\beta+2 }   \,dx\right)^{1/2} +  \left(\dashint_{ {\frac34} B}(|Du|^{\bz}|f|)^{2+\dz} \,dx\right)^{\frac1{2+\dz}}\right]\quad \forall B=B(z,R)\Subset U. 
\end{align}
 
%
 
 \end{enumerate}

\end{thm}

\begin{thm}Let $\Omega\subset\rn$ be a   domain with $n\ge3$,  $p(\cdot)\in C^{0,1}(\Omega;(1,\fz))$ and $f\in C^0(\Omega)$. 
For any   
 viscosity solution $u$  to the equation {\eqref{eq1f}}, the following hold.

\begin{enumerate}
\item[(i)] One has $D^2u\in L^2_\loc(\Omega^\ast_n)$.  Given any    subdomain $U\Subset\Omega_n^\ast$, one has  $D^2u\in L^{2+\dz}_\loc(U)$ for some $ \dz=\dz(p^U_\pm,n)>0$.   
  
\item[(ii)]   
Given any    subdomain $U\Subset\Omega$ and any $\bz\in(\bz_{\star}(n,p^U_+),\fz)\cap[0,\fz)$,
one has  $ |Du|^{ \bz}Du \in {W}^{1,2+\dz}_\loc (U)$ for some $\dz=\dz(p^U_\pm,n,\bz)>0$ and  a {\color{red}quantitative} upper bound  {\eqref{bound4}}
 with the constant $C(p^U_\pm,\bz)$ therein replaced by   $C(p^U_\pm,n,\bz)$. 
 
 \end{enumerate}
\end{thm}

Before sketching the ideas to prove Theorem 1.1--Theorem 1.4, we recall several well-known studies   in this direction in the literature. 

In the special case that the  exponent $p(\cdot)$ is a constant,   we denote $p(\cdot)$ by $p$. 
For $1<p<\fz$,     
the normalized $p$-Laplacian  reads as 
  $$ \Delta^N_{p }u= \Delta u{+}(p-2) \frac{\Delta_\fz  u}{|Du|^2}.$$
The  normalized $p$-Laplace 
equation   $ -\Delta^N_{p }u  =0 $ and its inhomogeneous equation
$-\Delta^N_pu=f$ with $f\in C^0(\Omega)$ are partially motivated by the tug-of-war;
see \cite{PS2008}. 
Viscosity solutions  to them in the sense of Crandall-Lions are well-defined. 
Multiplying    $ \Delta^N_{p }u   $ by $|Du|^{p-2}$,  
one gets formally the  $p$-{Laplacian} $\Delta_pu$, that is, 
  $$\Delta_pu:={\rm div}(|Du|^{p-2}Du)=|Du|^{p-2} \Delta^N_{p }u .$$
Thanks to the divergence structure in $\Delta_pu$, weak solutions  to    $\Delta_pu=0$ in $\Omega$ are well-defined in the Sobolev space $W^{1,p}(\Omega)$ and, usually, called as $p$-harmonic functions in $\Omega$. 
It was well-known that $u$ is a 
$p$-harmonic function  in $\Omega$ if and only if   $u $  is {a} viscosity solution to the equation  $\Delta_pu =0$ in $\Omega$,
and also if and only if $u $  is {a} viscosity solution  to the   equation $\Delta_p^Nu= 0$ in $\Omega$; see \cite{JLM2001,JJ2012,PS2008}. 
The existence and uniqueness of $p$-harmonic functions  are standard. 
For the $C^{1,\alpha}$-regularity of $p$-harmonic functions, we refer   to \cite{U1968,Uhlenbeck1977,Lewis1983,D1983,evans1982}. 
For their second order regularity, 
 when  $\gz>\bz_\star(n,p)$,    each $p$-harmonic function  $u$ satisfies  $D[|Du|^\gz Du]\in L^{2+\dz}_\loc$
for some $\dz>0$ depending on $n,p,\gz$; see for example {\color{red}\cite{dpzz19,ss2022}}.
 When $1<p<3+\frac2{n-2}$, since $\bz_\star(n,p)<0$, one may take $\gz=0$ here so to get  $D^2u\in L^{2+\dz}_\loc$, which was originally proved in \cite{MW1988} via the Cordes condition. 
In dimension 2,  $p$-harmonic functions enjoy    $C^{k,\alpha}$ and $W^{k+2,q}_\loc$-regularity 
with  sharp ranges of $k,\alpha,q$, and moreover,  {\color{red}for any $\bz>-1$ }the map {\color{red}$(x,y)\to  |Du|^\bz(u_x,-u_y)$} was shown to be quasiregular,
that is, {\color{red}$$|D[|Du|^\bz Du]|^2\le -K_{p,\beta}\det D[|Du|^\bz Du];$$ }
 see \cite{IM1989,Aron1989,BI1987,MW1988-0}. 
Regards  of the inhomogeneous   equation $-\Delta^N_pu=f$ with $f\in C^0(\Omega)$, 
we refer to   \cite{PS2008,RE2016} for the existence  of viscosity solutions.
The uniqueness  was established when 
  $f$ has constant sign. 
For the $C^{1,\alpha}$-regularity of viscosity solutions  we refer to Attouchi et al \cite{APR2017}.
When $1<p<3+\frac2{n-2}$,  viscosity solutions  were  shown in  \cite{AR2018} 
via Cordes conditions to enjoy $W^{2,2}_\loc$-regularity.



In general,  assume that  the exponent  $p(\cdot)\in C^0(\Omega;(1,\fz))$  is not a constant. 
Multiplying  $\Delta^N_{p(\cdot)}u$ by $|Du|^{p(\cdot)-2}$, one gets formally the  
strong $p(x)$-{Laplacian}   $\Delta^S_{p(\cdot)}u$  as introduced by Adamowicz-H\"ast\"o  \cite{AH2010, AH2011},   that is, 
 $$\Delta^S_{p(\cdot)}u = |Du|^{p(\cdot)-2}\Delta u+(p(\cdot)-2) |Du|^{p(\cdot)-2}\frac{\Delta_\fz u}{|Du|^2}.$$
Note that, under certain Sobolev regularity in $p(\cdot)$,
one would write  
  {\color{red}$\Delta^S_{p(\cdot)}$   } as 
\begin{equation}\label{eq1.2}
 \Delta^S_{p(\cdot)}u =-{\rm div}(|Du|^{p(\cdot)-2}Du)+|Du|^{p(\cdot)-2}\log(|Du|)Du\cdot Dp.
\end{equation} 
{\color{red}Since it consists of the divergence term $-{\rm div}(|Du|^{p(\cdot)-2}Du)$ and other term involving gradients, 
 weak solutions can be defined to the equation   $\Delta^S_{p(\cdot)}u =0$.  }
We say 
$u\in W_\loc^{1,p(\cdot)}(\Omega)$ is  a weak solution to the equation $\Delta^S_{p(\cdot)}u=0$ in $\Omega$ provided that 
$$
\int_\Omega |Du|^{p(\cdot)-2}Du\cdot D\phi\,dx+\int_\Omega |Du|^{p(\cdot)-2}\log(|Du|)Du\cdot Dp\phi\,dx=0\quad\forall  \phi\in W_0^{1,p(\cdot)}(\Omega). 
$$ 
Assuming   $p(\cdot)\in C^{0,1}(\Omega)$ in addition,
Adamowicz-H\"ast\"o \cite{AH2010} built up 
an existence  of  such weak solutions 
their interior H\"older continuity;
 and 
 moreover, in dimension $n=2$ they showed any such weak solution  $u$ 
  satisfies $D^2u\in L^2_\loc(\Omega)$ and 
  the map $(x_1,x_2)\to (u_{x_1},-u_{x_2})$ is  quasiregular in the sense 
that   
$$|D^2u|^2\le -\frac12\bigg(p(\cdot)-1+\frac1{p(\cdot)-1}\bigg)\det D^2u\quad \mbox{a. e. in}\ \Omega.$$
 When $p(\cdot)$ is log-H\"older continuous and satisfies $Dp\in L^n\log L^n$ or $Dp\in L^{q(x)}$ where $q\ge\max\{p,n\}+\dz$ for some $\dz>0$, Zhang-Zhou \cite{ZZ2012} obtained the $C^{1,\alpha}$-regularity of weak solutions to $\Delta^S_{p(\cdot)}u=0$.  The uniqueness of weak solutions to the equation $ \Delta^S_{p(\cdot)}u=0$ remains unknown. 
 
  Under  merely $p(\cdot)\in C^0(\Omega;(1,\fz))$, an existence of viscosity  solutions to $ \Delta^N_{p(\cdot)}u=0$ in $\Omega$ was established in \cite{AHP2017}. 
Assuming  $p(\cdot)\in C^{0,1}(\Omega)$ in addition,  Siltakoski \cite{SJ2018} proved that viscosity solutions to \eqref{eq1} actually coincide with weak solutions to $\Delta^S_{p(\cdot)}u=0$ in $\Omega$.  
 Siltakoski \cite{SJ2022} also obtained {\color{red} Lipschitz estimate and} $C^{1,\alpha}$-regularity of viscosity solutions to $-\Delta_{p(\cdot)}^Nu=f$ whenever $f\in C^0(\Omega)$. 
The uniqueness of viscosity  solutions to $ -\Delta^N_{p(\cdot)}u=0$ and  $ -\Delta^N_{p(\cdot)}u=f$   remains unknown. 

Under the assumption $p(\cdot)\in C^{0,1}(\Omega;(1,\fz))$,  in  Theorems 1.1 and {1.2} we establish a {\color{red}quantitative} second order regularity for viscosity  solutions to  the {\color{red}equation} \eqref{eq1}, and hence by Siltakoski \cite{SJ2018}, for 
 weak solutions to  the {\color{red}equation}  $\Delta^S_{p(\cdot)}u=0$. 
Moreover, in dimension $n=2$,   Theorems 1.1 gives the quasiregular property of  the map 
  $(x_1,x_2)\to |Du|^\bz (u_{x_1},-u_{x_2})$ for all  $\bz\ge 0$ and in dimension $n\ge3$,  Theorem {1.2}   bounds  $\sigma_2(D[|Du|^\bz Du] )$ via $|D[|Du|^\bz Du]|^2$ from below. These extend the above results  in dimension $n=2$ and $\bz=0$ obtained by  Adamowicz-H\"ast\"o \cite{AH2010}.   
In Theorems {1.3} and 1.4 we further build up  a {\color{red}quantitative} second order regularity for viscosity  solutions to  the inhmogenous {\color{red}equation} \eqref{eq1f}. 
Note that, when $f\ne 0$, Theorems {1.3} and 1.4  are  new.

Now we sketch the idea to prove Theorem 1.1--Theorem1.4, which is motivated by \cite{dpzz19,ss2022, SJ2022}.  Let $1< t_-\le p(x)\le t_+<\fz$ for all $x\in U$, 
and    $\bz\in(\bz_\star(n,t_+),\fz) $.   
Firstly,  we   prove   the following key pointwise inequality
\begin{align}\label{eq3.1in}
|D[(|Dv|^2+\ez)^{\bz/2}Dv]|^2\le C _\star (t_\pm,n,\beta) \sigma_2(D[(|Dv|^2+\ez)^{{\bz/2}}Dv])+ \wz C_\star (t_\pm,n,\beta)(|Dv|^2+\ez)^{\bz}(g-v)^2 \quad \mbox{in $B$} 
\end{align}
for any viscosity solution $v$ to the equation 
$$
-\Delta v-(p(\cdot)-2) \frac{\Delta_\fz v}{\left|D v\right|^{2}+\ez} = g-v\quad {\rm in}\ B,
$$
where $p$ is smooth in any given ball $B\Subset\Omega$. 
The proof is partially motivated by \cite{dpzz19,ss2022}; see Lemma  3.1 in Section 3 for more details.

 Next, given any viscosity solution $u$ to the equation \eqref{eq1f}, 
following \cite{SJ2022} we approximate $u$ by smooth functions $ \{v^\ez\}_{\ez\in(0,\ez_B)} $ for some small $\ez_B>0$. Here for any such $\ez$, $v^\ez$
 solves the Dirichlet probelm 
\begin{align}\label{neq4-1in}
\left\{
\begin{array} {ll}\displaystyle 
- \Delta v^\ez-(p^\ez(x)-2)\frac{\langle D^2v^\ez Dv^\ez,Dv^\ez\rangle}{|Dv^\ez|^2+\ez} =f^\ez+u^{0,\ez}-v^\ez\quad& \displaystyle\mbox{in $B$}\\[3mm]
  \displaystyle v^\ez=u\quad& \displaystyle \mbox{on $\partial B$}, 
\end{array}\right.
\end{align}
where   $f^\ez$ and $p^\ez,u^{0,\ez} $ are smooth approximation of 
$f, p, u$. Applying \eqref{eq3.1in} to $v^\ez$, using the divergence structure of $\sigma_2$ and sending $\ez\to0$, we are able to show that 
  $|Du|^{ \bz}Du\in W ^{1,2 }_\loc(U) $  with  a  {\color{red}quantitative} upper bound
 \begin{align}\label{bound4in}
\left\| D[ |Du | ^{\beta }Du ]\right\| _{L^{2 }(\frac 14B )}
 \le&  2^6
 C_\sharp (t_\pm,n,\bz)\left[\frac1{R } \left\| |Du  | ^{ \beta  }Du-\vec c\right\|_{L^2(\frac 12B )}+ \Big\| |Dv |^\beta f\Big\|_{L^{2 }(\frac 12 B )}\right]\quad \forall B=B(z,R)\Subset U 
\end{align}
 {  and also with a pointwise upper bound}
\begin{align}\label{eq3.1in-1}
 |D[|Du|^\bz Du]|^2\le C_\star (t_\pm,n,\beta) \sigma_2(D[ |Du|^ {\bz} Du])+   \wz C_\star (t_\pm,n,\beta)|Du|^ {2\bz}f^2 \quad \mbox{in $U$}. 
\end{align}
See Theorem 4.1 for details. 

Finally, applying  Sobolev's imbedding and Gehring's lemma, we get 
$|Du|^{ \bz}Du\in W ^{1,2+\dz }_\loc(U) $  and also the desired {\color{red}quantitative} upper bound;
see Lemma 4.1.  Theorem 1.1--Theorem 1.4 then follow  from Theorem 4.1 and Lemma 4.1.

We end the introduction by the following remark.

\begin{rem}\rm 
(i)   The restriction $\bz\in[0,\fz)$ appeared in Theorem   1.1--Theorem 1.4 is used to 
 guarantee the approximation argument in our approach. See Remark 4.1 (i) for  details. 
In Theorems 1.1 and 1.3, it  is not clear where 
we can relax  $\bz\in[0,\fz)$ to $\bz\in(-1,\fz)$. 

In  Theorem  {1.2} and  Theorem1.4, 
when $p^U_+> 3+\frac2{n-2}$, since $\bz_\star(n,p^U_+)>0$,
we know that $\bz\in(\bz_\star(n,p^U_+),\fz)$ gives $\bz\in[0,\fz)$ apparently; 
 while when $p^+\le 3+\frac2{n-2}$ ,  
  note that $\bz_\star(n,p^U_+)\le 0$, and 
    it is not clear for us whether  Theorem  {1.2} and Theorem 1.4  work for all $ \bz\in(\bz_\star(n,p^U_+),\fz)$.
 
If the uniqueness of viscosity solutions to \eqref{eq1} were proved,  
one  may obtain that Theorem 1.1  holds for all $\bz>-1$ and also 
Theorem 1.2 holds for all $\bz>\bz_\star(n,t_+)$ via adapting the current {\color{red}approximation} equations.  See Remark 4.1 (ii) for the reasons.

(ii) It is not clear whether   the assumption $p(x)\in C^{0,1}$ in   Theorem 1.1--Theorem 1.4  can be relaxed to merely 
$p(x)\in C^{0}$.  Recall  that $p(x)\in C^0$ is the most natural condition to define viscosity solutions to equations \eqref{eq1} and \eqref{eq1f}.
  The reason why we use $p(x)\in C^{0,1}$ is to gurantee a suitable convergence
  as did in the proof of Theorem 4.1 and also $Du\in L^\fz_\loc$.  See Remark 4.1 (iii) for the details. 


\end{rem}

{\color{red}
\begin{rem}\rm 
As referee pointed out, via using the known regularity theory of quasiregular mappings,  the range $\bz\in[0,\fz)$ in Theorem 1.1 can be relaxed to $\bz\in(-1,\fz)$ so that $|Du|^\bz Du\in W^{1,2}_\loc$. Note that 
the range $\bz\in(-1,\fz)$ here is sharp in the case that $p(x)$ is a constant; for details see \cite[Appendix]{dpzz22} and \cite{IM1989}. 
Precisely, by Theorem 1.1 with $\beta=0$, we know that    the mapping $x\to (u_{x_1},-u_{x_2})$ is locally quasiregular. Given any $\bz\in(0,1)$, 
denote by $\rho_\bz$ the corresponding radial stretching in plane, that is, 
$\rho_\bz(x)=|x|^\bz x$ for all $x\in\rr^2$. It is well-known that $\rho_\bz$ is a quasiconformal mapping; see \cite{SR1993}. 
The mapping  $x\to |Du|^\bz (u_{x_1},-u_{x_2})$ is  then  the composition of $\rho_\bz$ and   the mapping $x\to (u_{x_1},-u_{x_2})$, and hence is local quasiregular (see \cite[Theorem 9.4]{BI1983} and \cite[Section 2]{BI1987}), that is, $|Du|^\bz Du\in W^{1,2}_\loc(\Omega)$ and \eqref{bound3} holds. 
Moreover, in any domain $ U\Subset \Omega$, by Gehring's lemma (see \cite[Lemma 3.2.23]{MPS2000}),  
$|Du|^\bz Du\in W^{1,2+\dz}(U)$ 
for some $ \dz>0$ depending in $U,\bz$ 
 with 
$$\|D[|Du|^\bz Du] \|_{L^{2+\delta}(B)}\le C
  \|D[|Du|^\bz Du ] \|_{L^{2 }(2B)}\quad\forall \mbox{ball $B\Subset 2B\Subset U$}.$$ 
Thanks to \eqref{bound3} and the divergence of  $-\det D[|Du|^\bz Du]$,
  via integration by parts and also the H\"older inequality, one further conclude the quantitative upper bound \eqref{bound2}; we omit the details.

We further remark that Theorem 1.1 with $\bz=0$ was already obtained by Adamowicz-H\"ast\"o \cite{AH2010}.
Theorem 1.1 with $\bz>0$ can be also deduced from the above argument via radial stretchings $\rho_\bz$ and quasiregular theory. 
The contribution here is to give  a direct PDE's approch proof for Theorem 1.1 without  using radial stretchings and quasiregular theory.  
 \end{rem}

{\color{red} {\bf Acknowledgments}
We would like to thank  referees for several valuable suggestions and comments, in particular,  the approach in Remark 1.2, which  improved the presentation of this paper. 
} 
}

 \section{Prelimilaries}

In this section, we recall some  definitions and lemmas needed in this paper. 
We always let $\Omega\subset \rr^n$ be a bounded domain and $f\in C^0(\Omega)$. 
We recall from \cite{SJ2022,CIL92} the following definition of viscosity solutions  to the equation  \eqref{eq1f} and hence, 
 when $f=0$, to the    equation \eqref{eq1}.  
\begin{defn}\label{def2-1}  
 A lower semicontinuous  function $u:\Omega\to\rr$ is a viscosity supersolution to  $-\Delta^N_{p(\cdot)}u= f$ in $  \Omega$  
if for all $x_0\in\Omega$ and $\phi\in C^2(\Omega)$ such that $u-\phi$ attains a local minimum at $x_0$, one has
\begin{align*}
\left\{\begin{array}{ll}
\displaystyle 
-\Delta^N_{p(x_0)}\phi(x_0)\ge f(x_0) \quad &{\rm if}\ D\phi(x_0)\neq 0;\\
\displaystyle -\Delta\phi(x_0)-(p(x_0)-2)\lambda_{\min}(D^2\phi(x_0))\ge f(x_0) \quad &{\rm if}\  D\phi(x_0)=0\ {\rm and}\ p(x_0)\ge2;\\
\displaystyle -\Delta\phi(x_0)-(p(x_0)-2)\lambda_{\max}(D^2\phi(x_0))\ge f(x_0) \quad &{\rm if}\  D\phi(x_0)=0\ {\rm and}\ 1<p(x_0)<2,
\end{array}\right.
\end{align*}
where  $\lambda_{\min}(D^2\phi(x_0))$ and $\lambda_{\max}(D^2\phi(x_0))$ denote the minimum and maximum eigenvalue of the matrix 
$D^2\phi(x_0)$ respectively. 
 

An upper semicontinuous  function $u:\Omega\to\rr$ is a viscosity subsolution to $-\Delta^N_{p(\cdot)}u= f$ in $  \Omega$  if $-u$ is a viscosity supersolution.

A continuous function $u:\Omega\to\rr$ is a viscosity solution to $-\Delta^N_{p(\cdot)}u= f$ in $  \Omega$  if
$u$ is both viscosity supersolution and subsolution.

\end{defn}

We refer to \cite{kzz19,dpzz19,ss2022} for the following  fundamental inequality for $\Delta v\Delta_\fz v$. 
\begin{lem}\label{lem2.1}
 Let $U$ be a domain of  $\rr^n$. For any $v\in C^\fz(U)$, when $n=2$, one has
\begin{equation}\label{keyin0}
[|D^2v|^2-(\Delta v)^2]|Dv|^2 =2|D^2vDv|^2 -2\Delta_\fz v \Delta v;
\end{equation}
when $n\ge3$, one has
\begin{align}\label{keyin2}
[|D^2v|^2-(\Delta v)^2]|Dv|^4&\ge \frac n{n-1}|{D^2vDv}|^2|Dv|^2-\frac{n-2}{n-1}(\Delta v)^2|Dv|^4 \\
&-\frac2{n-1}\Delta_\fz v \Delta v |Dv|^2+\frac{n-2}{n-1}(|D^2vDv|^2|Dv|^2- {(\Delta_\fz v)^2} )\quad in\ U.\nonumber
\end{align}
\end{lem}

We also need the following algebraic structural identity for $\sigma_2$; 
see for example \cite[Lemma 2.4]{W2022}. 


\begin{lem}\label{lem2.2}
 Let   $\bz\in\rr$ and $\ez>0$.
For any   $v\in C^\fz(U)$,  we have
\begin{align}\label{id}
&\sigma_2(D[(|Dv|^2+\ez)^{{\bz/2}}Dv])\\
&\quad =\frac{1}{2}(|Dv|^2+\ez)^{\bz }[|D^2v|^2-(\bdz v)^2]
+\bz (|Dv|^2+\ez)^{{\bz-1}}[|D^2v Dv|^2-\bdz v\bdz_\fz v]\quad{\rm in}\ U.\nonumber
\end{align}
\end{lem}
 
The following identity of $\sigma_2$ follows from integration by parts; see \cite[Lemma 2.5]{W2022}.
\begin{lem}\label{lem2.3}
Let   $\bz\in\rr$ and $\ez>0$.
For any   $v\in C^\fz(U)$, any $\psi\in C_c^\fz(U) $ and any vector $\vec{c}\in\rr^n$, we have

\begin{align}\label{id2.4}
&\int_U\sigma_2(D[(|Dv|^2+\ez)^{{\bz/2}}Dv])\psi\,dx\\
&\quad =\int_U\sum\limits_{1\leq i<j\leq n}{\color{red}[}(|Dv|^2+\ez)^{\bz/2 }v_{x_i}-c_i][(|Dv|^2+\ez)^{\bz/2 }v_{x_j}]_{x_j}\psi_{x_i}\,dx\nonumber\\
&\quad\quad -\int_U\sum\limits_{1\leq i<j\leq n}{\color{red}[}(|Dv|^2+\ez)^{\bz/2 }v_{x_i}-c_i][(|Dv|^2+\ez)^{\bz/2 }v_{x_j}]_{x_i}\psi_{x_j}\,dx.\nonumber
\end{align}

%
%
%

\end{lem}

\begin{lem}{\color{red}Let   $\bz\in\rr$.}
For any  $v\in W^{1,2}(U)$ with  $|Dv|^\bz Dv\in W^{1,2}(U)$, any $\psi\in C_c^\fz(U) $ and any vector $\vec{c}\in\rr^n$, we have
{\color{red}
\begin{align}\label{id2.4-1}
 \int_U
\sigma_2(D[  |Dv| ^{\bz }Dv])\, \psi\,dx=&\int_U\sum\limits_{1\leq i<j\leq n} {\color{red}[} |Dv| ^{\bz }v_{x_i}-c_i]
[ |Dv| ^{\bz  }v_{x_j}]_{x_j}\psi_{x_i}\,dx \\ 
 & -\int_U\sum\limits_{1\leq i<j\leq n}  [ |Dv| ^{\bz }v_{x_i}-c_i]
[  |Dv| ^{\bz }v_{x_j}]_{x_i}\psi_{x_j}\,dx.\nonumber
\end{align}

%
%
%

\begin{proof}
For $\ez>0$, let $$F^\ez= (F^\ez_1,  F^\ez_2,...,F^\ez_n)=(|Dv|^\bz v_{x_1}-c_1,\cdots,|Dv|^\bz v_{x_n}-c_n )\ast\eta_\ez
 \in C^\fz_\loc(\Omega),$$  where $\eta_\ez$ is the standard smooth molliffier, in other words, $\eta_\ez(x)=\ez^{-n}\eta(x/\ez)$,
and $0\le \eta\in C_0^\fz(B(0,1))$ and $\int_\rn\eta(x)\,dx=1$. 
Thanks to the assumption $|Dv|^\beta Dv\in W^{1,2}_\loc(\Omega)$,  
one has $$\mbox{$F^\ez \to (|Dv|^\bz v_{x_1}-c_1,\cdots,|Dv|^\bz v_{x_n}-c_n )$ in $W^{1,2}_\loc(\Omega)$ as $\ez \to 0$.}$$  Note that $(F^\ez_k)_{x_ix_j}=(F^\ez_k)_{x_jx_i}$ for all possible $i,j,k.$ Therefore, for any $\psi\in C^\fz_c(\Omega )$, via 
 integration by parts  
one  obtains 
\begin{align}
\int_\Omega-\sigma_2\big(D\big[|Dv|^\beta Dv\big]\big)\psi\,dx&=
\lim_{\ez\to 0}\int_\Omega-\sigma_2(DF^\ez)\psi\,dx\nonumber\\
&=-\lim_{\ez\to0}\int_\Omega \sum\limits_{1\leq i<j\leq n}  \big[( F^\ez_i)_{x_i}
(F^\ez_j)_{x_j}  -  (F^\ez_i)_{x_j}
(F^\ez_j)_{x_i}\big]\psi \,dx \nonumber\\
&=\lim_{\ez\to0}\int_\Omega \sum\limits_{1\leq i<j\leq n}  \big[
 F^\ez_i(F^\ez_j)_{x_j}\psi_{x_i}+F^\ez_i
(F^\ez_j)_{x_jx_i}\psi  
-F^\ez_i
( F^\ez_j)_{x_i}\psi_{x_j}-F^\ez_i
( F^\ez_j)_{x_ix_j} \psi \big]\,dx \nonumber\\
&=\lim_{\ez\to0}\int_\Omega  \sum\limits_{1\leq i<j\leq n} \big[ F^\ez_i
(F^\ez_j)_{x_j}\psi_{x_i}  -
F^\ez_i
( F^\ez_j)_{x_i}\psi_{x_j}\big]\,dx \nonumber\\
&=\int_\Omega  \sum\limits_{1\leq i<j\leq n} \big[( |Dv | ^{{\bz }}v _{x_i}-c_i)
( |Dv | ^{{\bz }}v _{x_j} )_{x_j}\psi_{x_i}  -  ( |Dv | ^{{\bz }}v _{x_i}-c_i )
( |Dv | ^{{\bz }}v _{x_j})_{x_i}\psi_{x_j}\big]\,dx.\nonumber
\end{align}
We complete this proof.
\end{proof}

}

\end{lem}

We also need the following Gehring's lemma; see \cite{MPS2000}.
\begin{lem}\label{lem4} 
Let $U\subset\rn$ be  a subdomain and $1<r<s$. 
{\color{red}Suppose} that 
$ 0\le F \in L^r(U)$ and $0\le  G \in L^s(U)$ satisfy   
$$
\dashint_{B}(F(x))^r d x \leq K \left(\dashint_{2B} F(x) d x\right)^r+ K \dashint_{2B}(G(x))^r d x \quad\forall B\Subset 2B\Subset U
$$
for some constant $K\ge1$. 
Then there exists constants $\dz \in(0, s-r)$ and $C\ge1$,
 depending on depend only on $K, r, s$ and $n$, such that
$$
\left(\dashint_{B}(F(x))^{r+\dz} d x\right)^{1 / (r+\dz)} \leq C\left(\dashint_{2B}(F(x))^r d x\right)^{1 / r}+C\left(\dashint_{2B}(G(x))^{r+\dz} d x\right)^{1 / (r+\dz)} \quad\forall B\Subset 2B\Subset U. 
$$   

\end{lem}

\section {A key $\sigma_2$-estimate for approximation equations}

In this section,  
we establish the following key  pointwise $\sigma_2$-estimate  for smooth solutions to 
 approximation equations. 

\begin{lem}\label{lem3-1}
Let  $B=B(z,R)\subset\rn$  with $n\ge2$,
and  $p(\cdot)\in C^\fz(B)$ with $1<p(x)<\fz$ for all  $x\in B$ and $g\in C^\fz(B)$.   
For  any $\ez\in(0,1)$, let  $v\in C^2(  B)$ be any  solution to
\begin{align} \label{eq3.1}
-\Delta v-(p(\cdot)-2) \frac{\Delta_\fz v}{\left|D v\right|^{2}+\ez}+ v= g\quad {\rm in}\ B.
\end{align}

Given any $t_\pm\in(1,\fz)$ with  $ t_-\le p(x)\le t_+$ for all $x\in B$, 
and  any $\bz\in(\bz_\star(n,t_+),\fz) $,  we have  
\begin{align}\label{neq3.1}
|D[(|Dv|^2+\ez)^{\bz/2}Dv]|^2\le C_\star (t_\pm,n,\beta) \sigma_2(D[(|Dv|^2+\ez)^{{\bz/2}}Dv])+ \wz C_\star (t_\pm,n,\beta)(|Dv|^2+\ez)^{\bz}(g-v)^2 \quad \mbox{in $B$} 
\end{align}
and moreover,   for any $\phi\in C^\fz_c(B)$, 
\begin{align}\label{eq4.2}
&\int_{  B}|D[(|Dv|^2+\ez)^{\frac{\bz}2}Dv]|^2\phi^2 \,dx
\\
&\quad \le  C_\sharp( t_\pm,n,\bz)\left[ \int_{   B }|(|Dv|^2+\ez)^{\beta/2}Dv-\vec{c}|^2|D\phi|^2\,dx+ \int_{ B }(|Dv|^2+\ez)^{\bz}(g-v)^2\phi^2\,dx\right]   . \nonumber
\end{align}
Here constants $C_\star(t_\pm;n,\beta)\ge1  $, 
$\wz C_\star(t_\pm;n,\beta)\ge1  $ and
 $C_\sharp(t_\pm,n,\beta) \ge1 $ are  fixed in this paper. 
  \end{lem}
%
%
%
%

\begin{proof}[Proof of Lemma \ref{lem3-1}]
{\bf Proof of \eqref{eq4.2} via { \eqref{neq3.1}}.} Assuming that  {\eqref{neq3.1}} holds, we prove \eqref{eq4.2} as below.  
For any  $ \phi\in C^\fz_c(  U )$, applying {\color{red}\eqref{neq3.1}}  we have
\begin{align*}
\int_B&|D[(|Dv|^2+\ez)^{\beta/2}Dv]|^2\phi^2\,dx\\
&\le C_\star(t_\pm,n,\beta) \int_B\sigma_2(D[(|Dv|^2+\ez)^{{\bz/2}}Dv])\phi^2\,dx+\wz C_\star(t_\pm,n,\beta) \int_B(|Dv|^2+\varepsilon)^{\bz}(g-v)^2\phi^2\,dx.
\end{align*}
Using Lemma \ref{lem2.3}, one has 
\begin{align*}
\int_B&|D[(|Dv|^2+\ez)^{\beta/2}Dv]|^2\phi^2\,dx\\
&\le C_\star(t_\pm,n,\beta) C(n) \int_{B}|D[(|Dv|^2+\ez)^{{\bz/2}}Dv]||(|Dv|^2+\ez)^{\beta/2}Dv-\vec{c}||\phi D\phi|\,dx  \\
&\quad +\wz C_\star(t_\pm,n,\beta) \int_{B}(|Dv|^2+\varepsilon)^{\bz}(g-v)^2\phi^2\,dx .
\end{align*}
Applying Young's inequality, we obtain 
\begin{align*}
\int_B&|D[(|Dv |^2+\ez)^{\beta/2}Dv ]|^2\phi^2\,dx\\
&\le\frac12\int_{B}|D[(|Dv |^2+\ez)^{{\bz/2}}Dv ]|^2\phi^2dx\\
&\quad+C( t_\pm,\beta)[\int_{B}|(|Dv |^2+\ez)^{\beta/2}Dv -\vec{c}|^2|D\phi|^2\,dx+\int_{B}(|Dv|^2+\varepsilon)^{\bz}(g-v)^2\phi^2\,dx]. 
\end{align*}
which yields   \eqref{eq4.2} as desired.

{\bf Proof of  \eqref{neq3.1}.}  To prove \eqref{neq3.1}, 
observe that  
\begin{align*} |D[(|Dv|^2+\ez)^{\bz/2}Dv]|^2&=\left| (|Dv|^2+\ez)^{\bz/2}[D^2v- \bz\frac{ 
D^2 vDv\otimes Dv
}{  |Dv|^2+\ez  } ]\right|^2\\
&=(|Dv|^2+\ez)^{\bz}\left[|D^2v|^2-2\bz\frac{ 
|D^2 vDv|^2
}{  |Dv|^2+\ez  }+\bz^2\frac{ 
|D^2 vDv|^2 |Dv|^2
}{ ( |Dv|^2+\ez)^2  }\right]\\
&\le C(\bz) (|Dv|^2+\ez)^{\bz} |D^2v|^2.
\end{align*}
We only need to prove 
\begin{align}\label{eq3.1xx}
C(\bz)(|Dv|^2+\ez)^{\bz }|D^2v|^2\le C_\star (t_\pm,n,\beta) \sigma_2(D[(|Dv|^2+\ez)^{{\bz/2}}Dv])+ \wz C_\star (t_\pm,n,\beta)(|Dv|^2+\ez)^{\bz}(g-v)^2 \quad \mbox{in $B$.} 
\end{align}

To see \eqref{eq3.1xx}, fix any  $x\in U$. 
At such $x$,   Lemma \ref{lem2.2} gives 
\begin{align}\label{exx.1}
\sigma_2(D[(|Dv |^2+\varepsilon)^{\frac{\bz}{2}}Dv ])=&\frac{1}{2}(|Dv |^2+\varepsilon)^\bz[|D^2v |^2-(\Delta v )^2] \\
&+\bz(|Dv |^2+\varepsilon)^{\bz-1}[|D^2v  Dv |^2-\Delta v \Delta_\fz v].  \nonumber
\end{align}
  If $Dv(x)=0$, at such $x$, {\color{red}\eqref{eq3.1}} becomes
$$
-\Delta v=g-v,
$$ and 
\eqref{exx.1} reads as 
$$
\sigma_2(D[(|Dv |^2+\varepsilon)^{\frac{\bz}{2}}Dv ])= \frac{1}{2} \varepsilon^\bz
[|D^2v |^2-(\Delta v )^2] .
$$ 
It then follows that 
 \begin{align*}
 \ez^\bz |D^2v |^2=2\sigma_2(D[(|Dv|^2+\ez)^{{\bz/2}}Dv])+\ez^\bz(g-v)^2 
\end{align*}
which gives \eqref{eq3.1xx} obviously. 
 Below  we always assume that 
  $Dv(x)\neq0$.  We separate the proof for \eqref{eq3.1xx} into 5  steps.

   \vspace{0.3cm}
{\bf Step 1.} 
    For any $\eta>0$,  we multiply  both sides of \eqref{exx.1} by $$(n-1+2\eta) \frac{|Dv|^2}{(|Dv|^2+\ez)^{\bz+1}}$$ so to get 
\begin{align}\label{eq3.2-0}
{\bf K}:&=(n-1+2\eta)\frac{|Dv|^2}{|Dv|^2+\ez}\frac{\sigma_2(D[(|Dv|^2+\ez)^{{\bz/2}}Dv])}{(|Dv|^2+\ez)^\bz} \\
&\ge \eta \frac{|Dv|^2}{|Dv|^2+\ez}[|D^2v|^2-(\Delta v)^2]\nonumber\\
&\quad+\frac{n-1}2 \frac{|Dv|^2}{|Dv|^2+\ez}[|D^2v|^2-(\Delta v)^2]\nonumber\\
&\quad+(n-1+2\eta)\beta \frac{|Dv|^2}{|Dv|^2+\ez}\bigg[\frac{|D^2vDv|^2}{|Dv|^2+\ez}-\frac{\Delta v\Delta_\fz v}{|Dv|^2+\ez}\bigg].\nonumber
\end{align}

To bound the second term {\color{red}on the right-hand
side} of \eqref{eq3.2-0}, recall that \eqref{keyin2} gives 
  \begin{align*}
[|D^2v|^2-(\Delta v)^2]|Dv|^4&\ge \frac n{n-1}|{D^2vDv}|^2|Dv|^2-\frac{n-2}{n-1}(\Delta v)^2|Dv|^4 \\
&-\frac2{n-1}\Delta_\fz v \Delta v |Dv|^2+\frac{n-2}{n-1}(|D^2vDv|^2|Dv|^2- {(\Delta_\fz v)^2}).\nonumber
\end{align*}
Since $|Dv(x)|\ne0$, dividing  both sides   by $$\frac2{n-1} |Dv(x)|^2(|Dv(x)|^2+\ez), $$ we deduce that 
\begin{align} \label{eq3.2-1a}
\frac{n-1}2\frac{|Dv|^2}{|Dv|^2+\ez}[|D^2v|^2-(\Delta v)^2]\ge&\frac n2\frac{|D^2vDv|^2}{|Dv|^2+\ez}-\frac{n-2}2(\Delta v)^2\frac{|Dv|^2}{|Dv|^2+\ez} \\
&-\frac{\Delta_\fz v\Delta v}{|Dv|^2+\ez} +\frac{n-2}2\bigg[\frac{|D^2vDv|^2}{|Dv|^2+\ez}-\frac{(\Delta_\fz v)^2}{|Dv|^2(|Dv|^2+\ez)}\bigg]. \nonumber
\end{align}

Pluging \eqref{eq3.2-1a} in \eqref{eq3.2-0}, one has 
\begin{align}
{\bf K}&\ge \eta \frac{|Dv|^2}{|Dv|^2+\ez}[|D^2v|^2-(\Delta v)^2]\nonumber\\
&\quad+\frac n2\frac{|D^2vDv|^2}{|Dv|^2+\ez}-\frac{n-2}2(\Delta v)^2\frac{|Dv|^2}{|Dv|^2+\ez}\nonumber\\
&\quad-\frac{\Delta_\fz v\Delta v}{|Dv|^2+\ez}+\frac{n-2}2\bigg[\frac{|D^2vDv|^2}{|Dv|^2+\ez}-\frac{(\Delta_\fz v)^2}{|Dv|^2(|Dv|^2+\ez)}\bigg]\nonumber 
\\
&\quad+(n-1+2\eta)\beta\frac{|Dv|^2}{|Dv|^2+\ez}\bigg[\frac{|D^2vDv|^2}{|Dv|^2+\ez}-\frac{\Delta v\Delta_\fz v}{|Dv|^2+\ez}\bigg].\nonumber
\end{align} 
We reorganize the right hand side in above inequality {\color{red} to get} 
\begin{align}\label{eq3.2-1}
{\bf K}&\ge\eta \frac{|Dv|^2}{|Dv|^2+\ez}|D^2v|^2 \\
&\ \ +\bigg[(n-1+2\eta)\bz\frac{|Dv|^2}{|Dv|^2+\ez}+\frac n2+\frac {n-2}2\bigg] \frac{|D^2vDv|^2}{|Dv|^2+\ez}\nonumber\\
&\ \  -\bigg(\frac{n-2}2+\eta\bigg) \frac{|Dv|^2}{|Dv|^2+\ez} (\Delta v)^2\nonumber\\
& \ \ -\bigg[1+(n-1+2\eta){\beta} \frac{|Dv|^2}{|Dv|^2+\ez} \bigg]\frac{\Delta v\Delta_\fz v}{|Dv|^2+\ez} \nonumber\\
&\ \ - \frac{n-2}2 \frac{(\Delta_\fz v)^2}{|Dv|^2(|Dv|^2+\ez)}.\nonumber
\end{align}
 
   \vspace{0.3cm}
{\bf Step 2.} 
Since  
\begin{align}\label{eeq3.3-1}
\Delta v= -(g-v)-(p(x)-2) \frac{\Delta_\fz v  }{\left|D v\right|^{2}+\ez},
\end{align}
one has
$$\Delta_\fz v\Delta v= -(g-v)\Delta_\fz v-(p(x)-2) \frac{(\Delta_\fz v)^2  }{\left|D v\right|^{2}+\ez} $$
and 
\begin{align*}
(\Delta v)^2=  (g-v)^2 
+2(p(x)-2) (g-v) \frac{\Delta_\fz v  }{\left|D v\right|^{2}+\ez}+(p(x)-2)^2 \frac{(\Delta_\fz v)^2  }{(\left|D v\right|^{2}+\ez)^2}. 
\end{align*}
 Therefore, 
\begin{align}\label{eq3.2-1b}
{\bf K}
&\ge\eta \frac{|Dv|^2}{|Dv|^2+\ez}|D^2v|^2\\
&\ \ +\bigg[(n-1+2\eta)\bz\frac{|Dv|^2}{|Dv|^2+\ez}+\frac n2+\frac {n-2}2\bigg] \frac{|D^2vDv|^2}{|Dv|^2+\ez}\nonumber\\
&\ \ +\bigg\{-(\frac{n-2}2+\eta)(p(x)-2)^2\frac{|Dv|^2}{|Dv|^2+\ez}\nonumber\\
&\ \ \ \ \ \ \  \ \ \ \  \ \ \ \ +(p(x)-2)\bigg[1+(n-1+2\eta){\beta} \frac{|Dv|^2}{|Dv|^2+\ez} \bigg] -\frac{n-2}2\frac{|Dv|^2+\ez}{|Dv|^2}\bigg\} \frac{(\Delta_\fz v)^2}{(|Dv|^2+\ez)^2}\nonumber\\
&\ \ +{\bf K}_1, \nonumber
\end{align}
where 
\begin{align*}
{\bf K}_1=&-(\frac{n-2}2+\eta)\frac{|Dv|^2}{|Dv|^2+\ez}\bigg[(g-v)^2 
+2(p(x)-2) (g-v) \frac{\Delta_\fz v  }{\left|D v\right|^{2}+\ez}\bigg]\\
&+ \bigg[1+(n-1+2\eta){\beta} \frac{|Dv|^2}{|Dv|^2+\ez} \bigg]\frac{(g-v)\Delta_\fz v}{|Dv|^2+\ez}.\end{align*}

   \vspace{0.3cm}
{\bf Step 3.} 
To bound ${\bf K}_1$ from below, 
we further write 
\begin{align*}{\bf K}_1=&\bigg[-(n-2+2\eta)(p(x)-2)+(n-1+2\eta)\bz\bigg](g-v)\frac{\Delta_\fz v}{|Dv|^2+\ez} \frac{|Dv|^2}{|Dv|^2+\ez}\\
&+(g-v)\frac{\Delta_\fz v}{|Dv|^2+\ez} \\
&-(\frac{n-2}2+\eta) \frac{|Dv|^2}{|Dv|^2+\ez}(g-v)^2.
\end{align*}
 Recalling $\Delta_\fz v =\langle D^2v  Dv ,Dv \rangle$, by Cauchy-Schwarz's inequallity one gets 
$$|D^2v||Dv |^2\ge|\Delta_\fz v|.$$
Applying 
Young's inequality, one has 
\begin{align*} (g-v)\frac{\Delta_\fz v}{|Dv|^2+\ez}
&\ge -|g-v||D^2v|\frac{|Dv|^2}{|Dv|^2+\ez}\\
&\ge -
 \frac\eta 8\frac{|Dv|^2}{|Dv|^2+\ez}|D^2v|^2-
C(  \eta)\frac{|Dv|^2}{|Dv|^2+\ez}(g-v)^2 
\end{align*}
and,  by $t_-<p(x)<t_+$ in $B$,   also has
\begin{align*}
&\bigg[-(n-2+2\eta)(p(x)-2)+(n-1+2\eta)\bz\bigg](g-v)\frac{\Delta_\fz v}{|Dv|^2+\ez} \frac{|Dv|^2}{|Dv|^2+\ez}\\
&\ge - C(t_\pm,n,\beta)|g-v| |D^2u|\frac{|Dv|^4}{(|Dv|^2+\ez)^2}\\
&\ge -
 \frac\eta 8\frac{|Dv|^2}{|Dv|^2+\ez}|D^2v|^2-
C(t_\pm,n,\beta,  \eta)\frac{|Dv|^6}{(|Dv|^2+\ez)^3}(g-v)^2\\
&\ge -
 \frac\eta 8\frac{|Dv|^2}{|Dv|^2+\ez}|D^2v|^2-
C(t_\pm,n,\beta,  \eta) \frac{|Dv|^2}{|Dv|^2+\ez}(g-v)^2. 
\end{align*}
Therefore, 
\begin{align}\label{K1}
 {\bf K}_1 
&\ge -\frac{\eta}2 \frac{|Dv|^2}{|Dv|^2+\ez}|D^2v|^2-
C(t_\pm,n,\beta,\eta)  \frac{|Dv|^2}{|Dv|^2+\ez}(g-v)^2 
.
\end{align}

To bound the second term in the right hand side of \eqref{eq3.2-1b}, observe that
$$(n-1+2\eta)\bz\frac{|Dv|^2}{|Dv|^2+\ez}+\frac n2+\frac{n-2}2> [(n-1+2\eta)\bz+(n-1)]\frac{|Dv|^2}{|Dv|^2+\ez} .$$
If $\bz\ge0$,  for any $\eta>0$ we have $$(n-1+2\eta)\bz\frac{|Dv|^2}{|Dv|^2+\ez}+\frac n2+\frac{n-2}2>0$$
and
if $\bz_\star(n,t_+)<\bz<0$,  
for $0<\eta<\frac12(1+\bz)$ we have 
$$(n-1+2\eta)\bz+(n-1)=(n-1)(1+\bz)+2\eta\bz>(n-1)(1+\bz)-2\eta >(n-1)(1+\bz)-(1+\bz)\ge0, $$
and hence 
 $$(n-1+2\eta)\bz\frac{|Dv|^2}{|Dv|^2+\ez}+\frac n2+\frac{n-2}2>0.$$
Since the Cauchy-Schwarz inequality  yields
$$\frac{|D^2vDv|^2}{|Dv|^2+\ez}\ge\frac{(\Delta_\fz v)^2}{(|Dv|^2+\ez)^2}\frac{|Dv|^2+\ez}{|Dv|^2},$$
we have 
$$
\bigg[(n-1+2\eta)\bz\frac{|Dv|^2}{|Dv|^2+\ez}+\frac n2+\frac {n-2}2\bigg] \frac{|D^2vDv|^2}{|Dv|^2+\ez}\ge \bigg[(n-1+2\eta)\bz\frac{|Dv|^2}{|Dv|^2+\ez}+\frac n2+\frac {n-2}2\bigg] \frac{(\Delta_\fz v)^2}{(|Dv|^2+\ez)^2}\frac{|Dv|^2+\ez}{|Dv|^2}.$$
From this, \eqref{K1} and \eqref{eq3.2-1b} we deduce that 
\begin{align}\label{eeq3.6}
{\bf K}
&\ge\frac\eta2 \frac{|Dv|^2}{|Dv|^2+\ez}|D^2v|^2 +{\bf K}_2 \frac{(\Delta_\fz v)^2}{(|Dv|^2+\ez)^2} -C(t_\pm,n,\beta,\eta)\frac{|Dv|^2}{|Dv|^2+\ez}(g-v)^2,
\end{align}
where
\begin{align*}
{\bf K}_2=&\bigg[(n-1+2\eta)\bz\frac{|Dv|^2}{|Dv|^2+\ez}+\frac n2+\frac {n-2}2\bigg]\frac{|Dv|^2+\ez}{|Dv|^2}\\
&   -(\frac{n-2}2+\eta)(p(x)-2)^2\frac{|Dv|^2}{|Dv|^2+\ez}\\
& +(p(x)-2)\bigg[1+(n-1+2\eta){\beta} \frac{|Dv|^2}{|Dv|^2+\ez} \bigg] \\
& -\frac{n-2}2\frac{|Dv|^2+\ez}{|Dv|^2}.  
\end{align*}
 
{\bf Step 4.} 
Now we choose suitable $\eta>0$ so that  $
{\bf K}_2\ge0.
$
Set $$\tau=\frac{|Dv|^2+\ez}{|Dv|^2}.$$ 
Then ${\bf K}_2$ reads as 
\begin{align*}
{\bf K}_2=&\bigg[(n-1+2\eta)\bz\frac{1}{\tau}+\frac n2+\frac {n-2}2\bigg]\tau\\
&  -(\frac{n-2}2+\eta)(p(x)-2)^2\frac{1}{\tau}\\
&{+}(p(x)-2)\bigg[1+(n-1+2\eta)\beta \frac{1}{\tau} \bigg]\\
& -\frac{n-2}2\tau. 
\end{align*}
Multiplying both sides by $\tau$, we have 
$$\tau {\bf K}_2= \big[(n-1+2\eta)\bz\tau+\frac n2 \tau^2\big]-(\frac{n-2}2+\eta)(p(x)-2)^2+   (p(x)-2)\bigg[\tau+(n-1+2\eta)\beta   \bigg].$$
Observe  that $$\frac n2\tau^2=(n-1)\tau^2+\frac{2-n}2\tau^2$$ and   $$(p(x)-2)\tau=(p(x)-2)(n-1)\tau +(p(x)-2)(2-n)\tau.$$  
We have
\begin{align*}
\tau {\bf K}_2= &(n-1+2\eta)\bz\tau+(n-1+2\eta)(p(x)-2)\bz\\
&+(n-1)\tau^2+(p(x)-2)(n-1)\tau\\
&+\frac{2-n}2\tau^2+(p(x)-2)(2-n)\tau-\frac{n-2}2(p(x)-2)^2\\
&-\eta(p(x)-2)^2. 
\end{align*}
Observe that
$$(n-1+2\eta)\bz\tau+(n-1+2\eta)(p(x)-2)\bz=(n-1+2\eta)\bz(p(x)-2+\tau),$$
$$(n-1)\tau^2+(p(x)-2)(n-1)\tau=(n-1)\tau[p(x)-2+\tau],$$
and 
\begin{align*}
&\frac{2-n}2\tau^2+(p(x)-2)(2-n)\tau-\frac{n-2}2(p(x)-2)^2\\
&=-\frac{n-2}2\big[\tau^2+2(p(x)-2)\tau+(p(x)-2)^2\big]\\
&=-\frac{n-2}2(p(x)-2+\tau)^2.
\end{align*}
One writes 
\begin{align*}\tau {\bf K}_2=&(p(x)-2+\tau)\big[(n-1+2\eta)\bz+(n-1)\tau-\frac{n-2}{2}(p(x)-2+\tau) \big]-\eta(p(x)-2)^2\\
=& (n-1)(p(x)-2+\tau)\big[\frac{n-1+2\eta}{n-1}\bz+\tau-\frac{n-2}{2}\frac{p(x)-2+\tau}{n-1}\big]-\eta(p(x)-2)^2 .\end{align*}
Note that $t_-\le   p(x)  \le t_+$ in $B$ implies 
$$(p(x)-2)^2\le \max\{(t_--2)^2,(t_+-2)^2\} .$$
Since $\tau>1$ and $p(x) \ge t_->1$ in {$B$}, it follows that
$$ (p(x)-2+\tau)>t_--1>0.$$
Moreover,  by $\tau>1$ again one has 
\begin{align*}
&\frac{n-1+2\eta}{n-1}\bz+\tau-\frac{n-2}{2}\frac{p(x)-2+\tau}{n-1}\\
&\quad=\frac{n-1+2\eta}{n-1}\bz+\frac n{2(n-1)}\tau-\frac{n-2}2\frac{p(x)-2}{n-1}\\
&\quad>\frac{n-1+2\eta}{n-1}\bz+\frac n{2(n-1)}-\frac{n-2}2\frac{p(x)-2}{n-1}\\
&\quad=\bz+1-\frac{n-2}{2}\frac{p(x)-1}{n-1}+\frac{ 2\eta}{n-1}\bz.
\end{align*}
Since    $p(x) \le t_+$, 
it follows that 
\begin{align*}
 \frac{n-1+2\eta}{n-1}\bz+\tau-\frac{n-2}{2}\frac{p(x)-2+\tau}{n-1} 
>& \bz+1-\frac{n-2}{2}\frac{t_+-1}{n-1}+\frac{ 2\eta}{n-1} \bz= \bz-\bz_\star(n,t_+)  +\frac{ 2\eta}{n-1} \bz, 
\end{align*}
where we recall that 
$$\bz_\star(n,t_+)=-1+\frac{(n-2)}{2(n-1)}(t_+-1).$$
Thank to $\beta>\bz_\star(n,t_+)$, if  $\bz\ge0$, for any $ \eta>0$, one has 
$$\bz-\bz_\star(n,t_+)  +\frac{ 2\eta}{n-1} \bz >0, $$
and hence 
 \begin{align*}
\tau {\bf K}_2\ge (n-1)(t_--1)\big[\beta-\bz_\star(n,t_+)\big]-\eta \max\{(t_--2)^2,(t_+-2)^2\}. 
\end{align*} 
If $ \bz<0$,    we  choose    
$$\wz \eta_\star(t_\pm,n,\bz)=\frac12\min\left\{1+\bz ,\frac{n-1}{-2\bz}[\bz-\bz_\star(n,t_+)]\right\}.$$ Then for  any $0<\eta<\wz \eta_\star(t_\pm,n,\bz)$ we have 
 $$\bz-\bz_\star(n,t_+)  +\frac{ 2\eta}{n-1} \bz 
 \ge \frac12[\bz-\bz_\star(n,t_+) ]>0,$$
 and hence 
%
%
%
%
%
%
%
%
%
 \begin{align*}
\tau {\bf K}_2\ge\frac12 (n-1)(t_--1)\big[\beta-\bz_\star(n,t_+)\big]-\eta \max\{(t_--2)^2,(t_+-2)^2\}.
\end{align*} 

Finally, we choose $\eta_\star(t_\pm, n, \bz)\in(0,1/2)$ when $\bz\ge0$, and    $\eta_\star(t_\pm, n, \bz)\in(0,\wz \eta_\star(t_\pm,n,\bz))$ when
$
 \bz<0$,  such that
$$\eta_\star(t_\pm, n, \bz)\max\{(t_--2)^2,(t_+-2)^2\}<\frac12 (n-1)(t_--1)\big[\beta-\bz_\star(n,t_+)\big].$$ 
Consequently, for any $\eta\in(0,\eta_\star (t_\pm, n, \bz))$, we have $\tau {\bf K}_2\ge0$
and hence  $ {\bf K}_2\ge0$ as desired. 

   \vspace{0.3cm}
{\bf Step 5.} 
 Taking   $\eta=\frac12\eta_\star(t_\pm, n, \bz)$ in \eqref{eeq3.6}, due to ${\bf K}_2\ge0$ and \eqref{eeq3.6}   we get
\begin{align*}
{\bf K}
&\ge\frac{\eta_\star(t_\pm, n, \bz)}4\frac{|Dv|^2}{|Dv|^2+\ez}|D^2v|^2 -C(t_\pm,n,\beta) \frac{|Dv|^2}{|Dv|^2+\ez}(g-v)^2.
\end{align*}
Multiplying   both sides by $$\frac1{n-1+{\eta_\star(t_\pm, n, \bz)}}(|Dv|^2+\ez)^\beta\frac{|Dv|^2+\ez}{|Dv|^2},$$ 
considering the definition of ${\bf K}$ given in \eqref{eq3.2-0} one has
\begin{align*}
&\frac{4 [n-1+{\eta_\star(t_\pm, n, \bz)}]}{\eta_\star(t_\pm, n, \bz)}\sigma_2(D[(|Dv|^2+\ez)^{{\bz/2}}Dv])\\
&\quad= \frac4{\eta_\star(t_\pm, n, \bz)}   (|Dv|^2+\ez)^\beta\frac{|Dv|^2+\ez}{|Dv|^2}{\bf K}\\
&\quad\ge (|Dv|^2+\ez)^{\bz}|D^2v|^2-\frac4{\eta_\star(t_\pm, n, \bz)} C(t_\pm,n,\beta )(|Dv|^2+\ez)^{\bz}(g-v)^2.\nonumber
\end{align*}
Obviously \eqref{eq3.1xx} follows with the $C_\star(t_\pm,n,\beta )$ determind  completely by $t_\pm$, $n$ and $\bz$. 
 The proof is complete. 
\end{proof}

\section{Proofs of main  theorems}

%
%
%
%

 Firstly, we prove the following with the aid of key Lemma 3.1 and  an approximation procedure used in \cite{SJ2022}. 

\begin{thm} 
Let $\Omega\subset\rn$ be a   domain,  $p(\cdot)\in C^{0,1}(\Omega;(1,\fz))$ and $f\in C^0(\Omega)$.    
Let  $u$ be  any  viscosity solution  to the equation \eqref{eq1f}. 

Given any $ t_\pm\in(1,\fz)$     and any subdomain $U\Subset\Omega$ with $t_-\le    p(x)\le t_+$ in $U$, if 
 $\bz\in(\bz_\star(n,t_+),\fz) \cap[0,\fz)$,
 then one has 
  $|Du|^{ \bz}Du\in W_\loc^{1,2 }(U )$  with  a  {\color{red}quantitative} upper bound
 \begin{align}\label{bound4-1}
\left\| D[ |Du | ^{\beta }Du ]\right\| _{L^{2 }(\frac 12B )}
 \le&  2^6
 C_\sharp (t_\pm,n,\bz)\left[\frac1{R } \left\| |Du  | ^{ \beta  }Du-\vec c\right\|_{L^2(\frac 34B )}+ \Big\| |Dv |^\beta f\Big\|_{L^{2 }(\frac 34 B )}\right]\quad \forall B=B(z,R)\Subset U 
\end{align}
   and also with a pointwise upper bound
\begin{align}\label{eq3.1b}
 |D[|Du|^\bz Du]|^2\le C_\star (t_\pm,n,\beta) \sigma_2(D[ |Du|^ {\bz} Du])+   \wz C_\star (t_\pm,n,\beta)|Du|^ {2\bz}f^2 \quad \mbox{a.\,e. in $U$}. 
\end{align}
Here $C_\star(t_\pm,n,\beta)  $,  $\wz C_\star(t_\pm;n,\beta)  $ and $C_\sharp(t_\pm,n,\beta)  $  are same as  Lemma 3.1.

\end{thm}

\begin{proof}
Fix any ball $B=B(z,R)\Subset U$, and write $\ez_B:=\frac12\min\{ R,\dist(B,\partial U)\}$. 
Write  $\{f^\ez\}_{\ez\in(0,\ez_B)}$, $\{p^\ez\}_{\ez\in(0,\ez_B)}$ and $ \{u^{0,\ez}\}_{\ez\in(0,\ez_B)}$ as standard smooth modifications  
    of $f,p$ and $u$ respectively, that is, 
   $$ \mbox{$f^\ez=f\ast\eta_\ez,\, p^\ez=p\ast\eta_\ez$ and $u^{0,\ez}=u\ast\eta_\ez\quad\forall \ez\in(0,\ez_B)$},$$ 
   where     $ \eta\in C^\fz_c(B(0,1))$ is the standard smooth mollifier and $\eta_\ez(x)=\ez^{-n} \eta(z/\ez) $ for $\ez\in(0,\ez_B)$. 
   Observe that   
    $$\mbox{$f^\ez \to f, p^\ez\rightarrow p,u^{0,\ez}\rightarrow u$ uniformly in $\overline B$ as $\ez\rightarrow0$.}$$
For  $\ez\in(0,\ez_B)$,  one has 
$$t_-\le (p^\ez)^{ B}_-\le (p^\ez)^{ B}_+\le t_+,$$
and  also 
 $$\mbox{$||f^\ez||_{L^\fz(  B)}\le||f||_{L^\fz((1+\ez_B)B)}$,\  $||Dp^\ez||_{L^\fz( B)}\le||Dp||_{L^\fz((1+\ez_B)B)}$ \  \mbox{and}\
 $||u^{0,\ez}||_{L^\fz( B)}\le||u||_{L^\fz((1+\ez_B)B)}$.}$$

For $\ez\in(0,\ez_B)$, we consider the following approximation problem
\begin{align}\label{neq4-1}
\left\{
\begin{array} {ll}\displaystyle 
- \Delta w-(p^\ez(\cdot)-2)\frac{\Delta_\fz w}{|Dw|^2+\ez} =f^\ez+u^{0,\ez}-w\quad& \displaystyle\mbox{in $B$}\\[3mm]
  \displaystyle w=u\quad& \displaystyle \mbox{on $\partial B$}.
\end{array}\right.
\end{align}
Define  $$L_\ez w:=\Delta w+(p^\ez(\cdot)-2)\frac{\Delta_\fz w}{|Dw|^2+\ez}=\sum_{i,j=1}^n\left[\delta_{ij}+\frac{   w_{x_i}w_{x_j} }{|Dw|^2+\ez} \right]w_{x_i,x_j}.$$
Note that 
  $$\min(1,p^B_--1)|\xi|^2\le \sum_{i,j=1}^n
 \left[\delta_{ij}+(p^\ez(\cdot)-2)\frac{   \xi_i\xi_j }{|\xi|^2+\ez} \right]\xi_i\xi_j\le    (p^B_++1)|\xi|^2 \quad\forall \xi=(\xi_1,...,\xi_n)\in\rr^n\quad \mbox{in $B$}.$$ 
 Thus, $L_\ez$ is a  second order partial differnetial operator satisfying the   uniformly elliptic condition. 
 By the well-known standard theory for second order elliptic {\color{red}equations},
 for each $\ez\in (0,\ez_B)$, 
 the problem \eqref{neq4-1} has a solution   $ v^\ez\in C^\fz(B)\cap C(\overline B)$; see
 for example \cite[Theorem 15.18]{DT2001} and \cite[Theorem 6.17]{DT2001}.
Thanks to the  maximum principle (see {\cite[Theorem 10.3]{DT2001}}), we {\color{red}further} know that $\{v^\ez\}_{\ez\in(0,\ez_B)}$ are  bounded  uniformly  in $\ez\in(0,\ez_B)$; indeed, their upper bounds   depend only on $n,\, t_\pm,\, ||u||_{L^\fz((1+\ez_B)B)}$ and $||f||_{L^\fz((1+\ez_B)B)}$. 
By \cite[Proposition 4.14]{CC1995},   $\{v^\ez\}_{\ez\in(0,\ez_B)}$ are equicontinuous in $\overline B$; note that their modulus of continuity   depend only on $n,\, t_\pm,\, ||u||_{L^\fz((1+\ez_B)B)}$ and $||f||_{L^\fz((1+\ez_B)B)}$ and modulus of continuity of $u$ in $ (1+\ez_B) B$. Therefore, thanks to   the Ascoli-Arzela theorem, up to some subsequence  we have $v^\ez\rightarrow v\in C^0(\overline B)$ uniformly in $\overline B$.
By the stability principle of viscositly solutions (see \cite[Appendix]{SJ2022}), $v$ is a viscosity solution to 
\begin{equation}\label{vvv}\mbox{$-\Delta^N_{p(\cdot)}v=f+(u-v)$\quad in $B$;\quad   $v =u$ on $\partial B$}.\end{equation}  
Since the comparison principle holds for  \eqref{vvv}
  (see \cite{SJ2022}),  and both of $u$ and $v$ solve \eqref{vvv} in viscosity sense, 
  we have  $v=u$.
Moreover, by \cite{SJ2022} we know that
\begin{equation}\label{xx}{
\mbox{ $Dv^\ez\rightarrow Du$ in {\color{red}$C^{\alpha}(\frac 34B)$} and $
||Dv^\ez||_{L^\fz(\frac 34B)}\le C(n,t_{\pm} , \|Dp\|_{L^\fz((1+\ez_B)B)},||f||_{L^\fz((1+\ez_B)B)}, ||u||_{L^\fz((1+\ez_B)B)}){\color{red}.} $} }
\end{equation}

Since $t_-\le 
p^\ez (x)\le t_+$ in $  B$,   we are able to apply  Lemma 3.1 to $v^\ez$. 
In \eqref{eq4.2}, choosing $ \phi\in C^\fz_c( \frac34 B )$ so that
 $$  0\le \phi\le 1 \ \mbox{in}\   \frac34   B,\  \phi=1\  \mbox{in} \ \frac12 B\ \mbox{ and }  |D\phi|\le \frac 8R \ \mbox{ in} \ \frac34 B  ,$$ 
we  get 
\begin{align} \label{vepsupper}
&\int_{ \frac12B}|D[(|Dv^\ez|^2+\ez)^{ {\bz}/2}Dv^\ez]|^2 \,dx \\
&\le  2^6C_\sharp(t_\pm, n, \bz) \left [\frac1{R^2}\int_{ \frac34B }|(|Dv^\ez|^2+\ez)^{\beta/2}Dv^\ez-\vec{c}|^2\,dx +\int_{ \frac34 B }(|Dv^\ez|^2+\ez)^{\bz}(f^\ez+u^{0,\ez}-v^\ez)^2\,dx\right] \nonumber
\end{align}
 for any $\vec c\in\rr$. 
 
Thanks to \eqref{xx} and the fact that $u^{0,\ez}\to u$ and $ v^\ez\to u$ in
$ C^0(\overline B)$ as $\ez\to0$, we know that 
the right hand of \eqref{vepsupper} is   uniformly bounded in all  $\ez\in(0,\ez_B)$.
Thus  $$\mbox{$ (|Dv^\ez|^2+\ez)^{\frac\bz 2}Dv^\ez\in W^{1,2}(\frac12 B)$ uniformly in  all  $\ez\in(0,\ez_B)$. }$$
From the weak compactness of Sobolev spaces $W^{1,2}(\frac12B)$,  it follows that, 
up to some subsequence,  $ (|Dv^\ez|^2+\ez)^{\frac\bz 2}Dv^\ez$ converges to  
some funtion $\vec h$  in {\color{red}$L^2({\frac12}B)$} and weakly in $W^{1,2}( \frac12B)$. 
 Since $Dv^\ez\rightarrow Du$ in {\color{red}${\frac34}B$}   {\color{red}as given in \eqref{xx}},  we deduce that $\vec h= |Du|^\bz Du\in W^{1,2}(\frac12B)$.
{\color{red}From this and \eqref{xx} we conclude that }
$$\mbox{$(|Dv^\ez|^2+\ez)^{\frac\bz 2}Dv^\ez\rightarrow |Du|^\beta Du$  in $ L^2( \frac34 B)$ and weakly in $W^{1,2}(\frac 12B)$ as $\varepsilon\rightarrow 0$.  }$$
 Letting $\ez\to0$ in \eqref{vepsupper}, one has 
 \begin{align*}
\int_{\frac 12B }|D[  |Du| ^{\beta }Du ]|^2\,dx
\le &2^6C_\sharp(t_\pm,n,\bz)\left[\frac1{R^2} \int_{\frac 34B }   ||Du |^{\beta } Du -\vec c |^2 \,dx+ \int_{\frac 34B }|Du|^{2\bz}f^2\,dx\right] 
\end{align*}
for any $\vec c\in\rn$. 
 Thus  \eqref{bound4-1} follows. Taking over all balls $B\Subset U$, one has  $ |Du| ^{\beta }Du\in W^{1,2}_\loc (U)$.

Finally,  we prove  \eqref{eq3.1b}. 
For any $\phi\in C^\fz_c(\frac12 B)$, multiplying both sides of  \eqref{neq3.1} by $\phi^2$ and integrating,   one    has 
\begin{align}\label{eq4.xx1}
 \int_{{\frac12  }B}|D[(|Dv^\ez|^2+\ez)^{\frac{\bz}2}Dv^\ez]|^2\phi^2 \,dx&\le  C_\star (t_\pm,n,\beta)\int_{{\frac12  }B}
\sigma_2(D[(|Dv^\ez|^2+\ez)^{{\bz/2}}Dv])\, \phi^2\,dx\\
 &\quad + \wz C_\star (t_\pm,n,\beta)\int_{{\frac12  }B}(|Dv^\ez|^2+\ez)^{\bz}(f+u^{0,\ez}- v^\ez)^2\phi^2\,dx  \nonumber
\end{align}
Since 
$D[(|Dv^\ez|^2+\ez)^{\frac{\bz}2}Dv^\ez]\to D[  |Du| ^{\beta }Du ]$ weakly in $L^2({\frac12} B)$ one has 
$$
\int_{{\frac12  }B}|D[  |Du| ^{\beta }Du ]|^2\psi\,dx\le\lim_{\ez\to0}\int_{{\frac12  }B}|D[(|Dv^\ez|^2+\ez)^{\frac{\bz}2}Dv^\ez]|^2\psi \,dx.$$
Noting that  $(|Dv^\ez|^2+\ez)^{\frac{\bz}2}Dv^\ez \to |Du| ^{\beta }Du$
in $L^2({\frac34} B)$ and weakly in $W^{1,2}({\frac12} B)$,
by \eqref{id2.4} and \eqref{id2.4-1} one has 
\begin{align*} 
&\int_{{\frac12  }B}\sigma_2(D[(|Dv^\ez|^2+\ez)^{{\bz/2}}Dv^\ez])\phi\,dx\\
&\quad =\int_{{\frac12  }B}\sum\limits_{1\leq i<j\leq n} {\color{red}[}(|Dv^\ez|^2+\ez)^{\bz/2 }v^\ez_{x_i}-c_i][(|Dv^\ez|^2+\ez)^{\bz/2 }v^\ez_{x_j}]_{x_j}\phi_{x_i}\,dx\nonumber\\
&\quad\quad -\int_{{\frac12  }B}\sum\limits_{1\leq i<j\leq n} {\color{red}[}(|Dv^\ez|^2+\ez)^{\bz/2 }v^\ez_{x_i}-c_i][(|Dv^\ez|^2+\ez)^{\bz/2 }v^\ez_{x_j}]_{x_i}\phi_{x_j}\,dx\\
&\quad\to \int_{{\frac12  }B}\sum\limits_{1\leq i<j\leq n} {\color{red}[} |Du| ^{\bz }u_{x_i}-c_i]
[ |Du| ^{\bz  }u_{x_j}]_{x_j}\phi_{x_i}\,dx\nonumber\\
&\quad\quad -\int_{{\frac12  }B}\sum\limits_{1\leq i<j\leq n} {\color{red}[}  |Du| ^{\bz }u_{x_i}-c_i]
[  |Du| ^{\bz }u_{x_j}]_{x_i}\phi_{x_j}\,dx\\
&\quad= \int_{{\frac12  }B}
\sigma_2(D[  |Du| ^{\bz }Du])\, \phi\,dx. 
\end{align*}
Moreover,   noting $u^{0,\ez}- v^\ez\to u-u= 0$ in $C^0({\frac12} \overline B)$, by $\bz\ge 0$ we have 
\begin{align} \label{conf}
\int_{{\frac12  }B}(f^\ez+u^{0,\ez}- v^\ez)^2(|Dv^\ez|^2+\ez)^{\bz}\phi\,dx \to  
\int_{{\frac12  }B} f^2 |Du| ^{\bz}\phi\,dx.\end{align}

Combinging these, letting $\ez\to0$ in \eqref{eq4.xx1}, one has   
\begin{align}\label{eq4.x1-1}
 \int_{{\frac12  }B} |D[|Dv|^\bz Du]|^2\phi^2\,dx 
 & \le C_\star (t_\pm,n,\beta)\int_{{\frac12  }B}\sigma_2(D[ |Du| ^{{\bz }}Du])\phi^2\,dx\\
 &\quad\quad+\wz C_\star (t_\pm,n,\beta) \int_{{\frac12  }B}(|Du|^{2\bz}f ^2\phi^2\,dx \quad\forall \phi\in C^\fz_c({\frac12  }B). \nonumber
\end{align}
This implies that 
 $$|D[|Dv|^\bz Du]|^2 
   \le C_\star (t_\pm,n,\beta) \sigma_2(D[ |Du| ^{{\bz }}Du]) +\wz C_\star (t_\pm,n,\beta)  (|Du|^{2\bz}f ^2\ \mbox{a.\,e. in}\ \frac14 B .$$
  Taking over all balls $B\Subset U$, we obtain \eqref{eq3.1b} as desired.  
\end{proof}

Next we use the Gehring lemma to improve the integrability of  $D[ |Du | ^{\beta }Du]$. 

\begin{lem}\label{lem4-2} 
Let $\Omega\subset\rn$ be a   domain,  $p(\cdot)\in C^{0,1}(\Omega;(1,\fz))$ and $f\in C^0(\Omega)$.    
Let  $u$ be  any  viscosity solution  to the equation \eqref{eq1f}.

Given any $ t_\pm\in(1,\fz)$  and  any subdomain $U\Subset\Omega$ with $t_-\le p(x)\le t_+$ in $U$, if $\bz\in(\bz_\star(n,t_+),\fz) \cap[0,\fz)$,
 then   
    $     |Du | ^{\beta }Du\in W^{1,2+\dz}_\loc(U)  $ 
     for some constant $ \dz=\dz(t_\pm n, \bz )>0 $ with 
     \begin{align}\label{xxx1}
&\left(\dashint_{\frac 14B}|D(|Du|^\beta Du)|^{2+\dz}\,dx\right)^{\frac 1{ 2+\dz }}\nonumber\\
&\le C(t_\pm,n,  \beta ) \left[\frac1{R } \left(\dashint_{ {\frac34} B }   ||Du |^{\beta } Du -\vec c |^2 \,dx\right)^{1/2} +  \left(\dashint_{ {\frac34} B}(|Du|^{\bz}|f|)^{2+\dz} \,dx\right)^{\frac1{2+\dz}}\right] \quad\forall B=B(z,R)\Subset U.  
\end{align}
     
%
 Note that  $\dz(t_\pm, n, \bz)>0$ and  $C(t_\pm, n, \bz)\ge1$ are constants depending only on $t_\pm, n$ and $\bz$. 
 \end{lem}
\begin{proof} Note that $Du\in L^\fz(U)$.  
By Theorem 4.1, 
   $     |Du | ^{\beta }Du\in W^{1,2}_\loc (U)  $.
For any ball $B=B(z,R)\Subset U$,   \eqref{bound4-1} implies   
    \begin{align}
    \dashint_{{\frac 14}B }|D[  |Du| ^{\beta }Du ]|^2\,dx
\le & 2^6C_\sharp(t_\pm, n ,\bz)\left[\frac1{R^2} \dashint_{ {\frac12} B }   ||Du |^{\beta } Du -\vec c_{{\frac12} B} |^2 \,dx
+ \dashint_{ {\frac12} B}|Du|^{2\bz}f^2\,dx\right],
\end{align}
where  $\vec c_{{\frac12} B}=\dashint_{ {\frac12} B}|Du |^{\beta } Du\,dx $.
   Applying the Sobolev-Poincar\' e 
inequality,  we have
$$
\frac1{R^2} \dashint_{ {\frac12} B }   ||Du |^{\beta } Du -\vec c |^2 \,dx\le C(n)
\Big(\dashint_{ {\frac12} B }|D(|Du|^\beta Du)|^{\frac{2n}{n+2}}\,dx\Big )^\frac{n+2}{n}.$$
Thus 
\begin{align*}
\dashint_{{\frac14} B }|D(|Du|^{\beta}Du)|^2\,dx\le &C_\sharp ( t_{\pm},n,\beta) C(n)\left[
\Big(\dashint_{ {\frac12} B }|D(|Du|^\beta Du)|^{\frac{2n}{n+2}}\,dx\Big )^\frac{n+2}{n}+ \dashint_{ {\frac12} B }|Du|^{2\bz}f^2\,dx\right]. 
\end{align*}
Note that   $|Du|^{ \bz}|f|\in L^\fz(U)$ and hence $|Du|^{ \bz}|f|\in L^s(U)$ for all $s>1$. 
Applying Lemma \ref{lem4} with $r= \frac {n+2}n$, $F=|D(|Du|^{\beta}Du)|^{\frac{2n}{n+2}}$ and $G= [ |Du|^{ \bz}|f|]^{\frac{2n}{n+2}}$, there exists   $\dz=\dz(t_\pm,n, \bz)>0$ such that 
\begin{align*}
\left(\dashint_{{\frac14} B}|D(|Du|^\beta Du)|^{2+\dz}\,dx\right)^{\frac 1{ 2+\dz }}
&\le  C( t_{\pm},n,\beta)\left[
\left(\dashint_{ {\frac12} B }|D(|Du|^\beta Du)|^2\,dx\right )^{\frac12}+ \left(\dashint_{ {\frac12} B}(|Du|^{\bz}|f|)^{2+\dz} \,dx\right)^{\frac1{2+\dz}}\right]. 
\end{align*}
for any ball $B=B(z,R)\Subset U$.

Applying  \eqref{bound4} again   we  have 
\begin{align*}
&\left(\dashint_{\frac 14B}|D(|Du|^\beta Du)|^{2+\dz}\,dx\right)^{\frac 1{ 2+\dz }}\\
&\le  C(t_\pm,n,  \beta ) \left[\frac1{R } \left(\dashint_{ {\frac34} B }   ||Du |^{\beta } Du -\vec c |^2 \,dx\right)^{1/2}{+\left(\dashint_{\frac34B}(|Du|^\beta f)^2\,dx\right)^{1/2}}+  \left(\dashint_{ {\frac12} B}(|Du|^{\bz}|f|)^{2+\dz} \,dx\right)^{\frac1{2+\dz}}\right] 
\end{align*}
for any ball $B=B(z,R)\Subset U$, which together with 
 the Cauchy-Schwarz inequality yields   \eqref{xxx1} as desired.
\end{proof}

We are ready to prove Theorem 1.1-theorem 1.4. 

\begin{proof}[Proofs of Theorems 1.1 and 1.3.]
In dimension $n=2$,  one has $\bz_{\star}(2,t_+)=-1$ for all $t_+>1$.
Hence, 
 $\bz\ge0$ always implies that $\bz>\bz_{\star}(2,t_+)$.  Take $t_\pm= p^U_\pm$.
 Applying Theorem 4.1 and Lemma 4.1  with $f=0$, we get the desired Theorem 1.1. 
 Applying Theorem 4.1 and Lemma 4.1   with general $f\in C^0(\Omega)$, we get the desired Theorem 1.3.
\end{proof}

\begin{proof}[Proofs of Theorems 1.2 and 1.4.] Take $t_\pm= p^U_\pm$. 
  Applying Lemma 4.1 with $f=0$, we get the desired Theorem 1.2 (ii) directly. 
   Applying Lemma 4.1 with general $f\in C^0(\Omega)$, we get the desired Theorem 1.4 (ii) directly. 
  
  To see Theorem 1.2 (i), note that for any $U\Subset\Omega_n^\ast$, one has 
  $1<p^U_-\le p^U_+< 3+\frac2{n-2}$. Thus  
 $$ \bz_\star(n,p^U_+)< \bz_\star(n,3+\frac2{n-2})= -1+\frac{ n-2 }{n-1}(3+\frac2{n-2}-1)=0.$$
Applying Theorem 1.2 (ii) to any $U\Subset\Omega_n^\ast$, taking  $\bz=0$ we get   Theorem 1.2 (i). 
Applying Theorem 1.4 (ii) to any $U\Subset\Omega_n^\ast$, taking  $\bz=0$ we get   Theorem 1.4 (i). 
\end{proof}

We end this section by the following remark.

\begin{rem}\rm 
(i) 
In Theorem 4.1,   the adiditonal assumption $\bz\ge0$  is used to guarantee the convergnece \eqref{conf}. 
If $p^ U_+\ge 3+\frac2{n-2}$ we must have $t_+\ge 3+\frac2{n-2}$ and hence 
 $$\bz_\star(n,t_+) \ge\bz_\star(n,3+\frac2{n-2}) =0.$$
 So $\bz>\bz_\star(n,t_+)$ implies $ \bz\ge0$. 
If  $p^ U_-<3+\frac2{n-2}$, it may happen that  $p^U_+\le t_+<3+\frac2{n-2}$, which implies   
 $$\bz_\star(n,t_+)<\bz_\star(n,3+\frac2{n-2})=0.$$   
 When $\bz_\star(n,t_+)<\bz<0,$
  it is not clear whether the convergence \eqref{conf} holds or not. 
  Indeed, in the case $f=0$ and hence $ f^\ez=0$,   
  \eqref{conf} becomes 
\begin{equation}\label{convf0}\int_{\frac12 B }(|Dv^\ez|^2+\ez)^{\bz}( u^{0,\ez}-v^\ez)^2\,dx\to 0.
\end{equation}
Since the convergence rate of $u^{0,\ez}-v^\ez\to0$ is not clear,  we {\color{red}do not  know whether} \eqref{convf0} holds or not. 
In the case $f\ne 0$, say $ f=1$,  to get 
\eqref{conf}, we need apriori that  $|Du|^{2\bz}\in L^1_\loc$, which we do not know yet. 

(ii)  Recall that  the uniquess of viscosity solutions to {\color{red}equations} \eqref{eq1} and \eqref{eq1f} remains open. 
This is the reason that, in the proof of  Theorem 4.1, we use  approximation equations \eqref{neq4-1}  so that 
their smooth solutions $v^\ez$ approximate the given solution $u$ to the equation \eqref{eq1f} and, when $f=0$, to the {\color{red}equation} \eqref{eq1}.

In the case $f=0$, suppose that the uniquess of viscosity solutions to equations \eqref{eq1} holds, we {\color{red}may} remove the additional assumption $\bz\ge0$ in Theorem 4.1 and hence Lemma 4.1 via a different approximation argument as below. 
For any $\ez\in(0,\ez_B)$, instead of \eqref{neq4-1}, let $w^\ez$ be the smooth solution  to the equation 

\begin{align}\label{appro4-1}
\left\{
\begin{array} {ll}\displaystyle 
- \Delta w^\ez-(p^\ez(x)-2)\frac{\Delta_\fz w^\ez }{|Dw^\ez|^2+\ez} =0& \displaystyle\mbox{in $B$}\\[3mm]
  \displaystyle w^\ez=u\quad& \displaystyle \mbox{on $\partial B$}.
\end{array}\right.
\end{align}
 The uniquess for viscosity solutions to {\color{red}equations} \eqref{eq1} will guarantee that $w^\ez\to u$ in $ C^0(B)$ as $ \ez\to0$. 
 For any $\bz>\bz_\star(n,t_+)$, which maybe less than $0$, 
 instead of  \eqref{vepsupper}, we will get 
\begin{align} \label{vepsupper-1}
&\int_{ \frac12B}|D[(|Dv^\ez|^2+\ez)^{\frac{\bz}2}Dv^\ez]|^2 \,dx \le  2^6C_\sharp(t_\pm, n, \bz)  \frac1{R^2}\int_{ \frac34B }|(|Dv^\ez|^2+\ez)^{\beta/2}Dv^\ez-\vec{c}|^2\,dx 
\end{align}
 for any $\vec c\in\rr$.  Sending $\ez\to0$  and repeating above argument,
  we will get  \eqref{bound4-1} and \eqref{eq3.1b} with $f=0$ therein.  Repeating the proof of
  Lemma 4.1 we also have \eqref{xxx1} with $f=0$. 
 
 (iii) The assumption $p(\cdot
 )\in C^{0,1}(\Omega)$ in Theorem 4.1 and hence in Lemma 4.1  is  used to get the uniform 
 $C^{1,\alpha}$-regualrity of $\{v^\ez\}$ so to get the convergence of $Dv^\ez\to Du$ locally uniformly.  
 Moreover, $p(\cdot)\in C^{0,1}(\Omega)$ also implies that $Du\in L^\fz_\loc(\Omega)$. 
 
 Under  merely $p(\cdot)\in C^{0}(\Omega)$, which is the   most natural assumption to define viscosity solutions, 
it is not clear whether  the above argument can be adaped so to get the same conclusions as Theorem 4.1.

\end{rem}

 \par \indent  Yuqing Wang 
 \par \indent  Department of Mathematics  
  \par \indent  Beihang University
   \par \indent  Changping District Gaojiaoyuan Road South 3rd Street No.9
   \par \indent  Beijing 102206
   
    \par \indent  P.R. China 
 
  \par \indent  Email: yuqingwang@buaa.edu.cn,
 
  \par \indent 
  
 Yuan Zhou

  \par \indent School of  Mathematical Sciences 
  
 \par \indent  Beijing Normal University 
 
  \par \indent Haidian District Xueyuan Road Xinjiekou Waidajie No. 19 
  
  \par \indent Beijing 100875 
  
  \par \indent  P. R. China. 
 
  \par \indent 
 Email:  yuan.zhou@bnu.edu.cn


\begin{thebibliography}{30}
 
\bibitem{AH2010}
T. Adamowicz and  P. H\"ast\"o,   \emph{
Mappings of finite distortion and PDE with nonstandard growth.}
Int. Math. Res. Not. IMRN 2010, no. 10, 1940-1965.

\bibitem{AH2011}
T. Adamowicz and  P. H\"ast\"o,   \emph{
Harnack's inequality and the strong $p(\cdot)$-Laplacian.  } 
J. Differential Equations 250 (2011), no. 3, 1631-1649.

{
\bibitem{Aron1989} 
G. Aronsson,
\emph{Representation of a $p$-harmonic function near a critical point in the plane.} Manuscripta Math.66(1989), no.1, 73-95.

}

\bibitem{AHP2017}
A. Arroyo,  J. Heino and M. Parviainen,  \emph{
Tug-of-war games with varying probabilities and the normalized $p(x)$-Laplacian. } 
Commun. Pure Appl. Anal. 16 (2017), no. 3, 915-944.

\bibitem{APR2017}
A. Attouchi, M. Parviainen and E. Ruosteenoja, 
\emph{$C^{1,\alpha}$ regularity for the normalized p-Poisson problem. }
J. Math. Pures Appl. (9) 108 (2017), no. 4, 553-591.

\bibitem{AR2018}
A. Attouchi and E. Ruosteenoja, 
\emph{Remarks on regularity for p-Laplacian type equations in non-divergence form. }
J. Differential Equations 265 (2018), no. 5, 1922-1961.

\bibitem{BD2012}  
I. Birindelli and F. Demengel,  \emph{ Regularity for radial solutions of degenerate fully nonlinear equations.} Nonlinear Anal, 75(17): 6237-6249, 2012.
{\color{red}
\bibitem{BI1983}
B. Bojarski, and  T. Iwaniec, \emph{Analytical foundations of the theory of quasiconformal mappings in $\rr^n$.} Ann. Acad. Sci. Fenn. Ser. AI Math. 8(1983), 257-324.
}

\bibitem{BI1987}
B. Bojarski and T. Iwaniec, \emph{$p$-harmonic equation and quasiregular mappings.} Partial differential equations (Warsaw, 1984), 25-38, Banach Center Publ., 19, PWN, Warsaw, 1987.



\bibitem{CC1995}
L. A. Caffarelli and X. Cabr\'e, \emph{Fully nonlinear elliptic equations.} American Mathematical Society Colloquium Publications, 43. American Mathematical Society, Providence, RI, 1995. vi+104 pp.

\bibitem{CIL92}
M. G. Crandall, H. Ishii and P.-L. Lions, \emph{User's guide to viscosity solutions of second order partial differential equations.}
Bull. Amer. Math. Soc. (N.S.) 27 (1992), no. 1, 1-67.

\bibitem{DF2010}
G. D$\rm{\acute a}$vila, P. Felmer and A. Quaas, 
\emph{Harnack inequality for singular fully nonlinear operators and some existence results. }
Calc. Var. Partial Differential Equations 39 (2010), no. 3-4, 557-578.

\bibitem{dpzz19}
H. Dong, F. Peng, Y. R. Zhang and Y. Zhou, \emph{Hessian estimates for equations involving $p$-Laplacian via a fundamental inequality.} Adv. Math. 370 (2020), 107212, 40 pp.
{\color{red}
\bibitem{dpzz22}
H. Dong, F. Peng, Y. R. Zhang and Y. Zhou, \emph{Jacobian determinants for nonlinear gradient of planar $\fz$-harmonic functions and applications.} Preprint, arXiv:2209.02659.}

\bibitem{D1983}
E. DiBenedetto, \emph{$C^{1+\alpha}$ local regularity of weak solutions of degenerate elliptic equations.} Nonlinear Anal. 7 (1983), no. 8, 827-850.

\bibitem{evans1982}
L. C. Evans, \emph{A new proof of local $C^{1,\alpha}$ regularity for solutions of certain degenerate elliptic p.d.e.} J. Differential Equations 45 (1982), no. 3, 356-373.

\bibitem{FZ2021}
Y. Fang and C. Zhang,  \emph{Gradient H\"older regularity for parabolic normalized p(x,t)-Laplace equation.} J. Differential Equations 295 (2021), 211-232.

\bibitem{G1973}
F. W. Gehring, \emph{The $L^p$-integrability of the partial derivatives of a quasiconformal mapping.} Acta Math. 130 (1973), 265-277. 

\bibitem{GM1983}
M. Giaquinta, \emph{Multiple integrals in the calculus of variations and nonlinear elliptic systems.} Annals of Mathematics Studies, 105. Princeton University Press, Princeton, NJ, 1983. vii+297 pp.

\bibitem{DT2001}
D. Gilbarg and N. S. Trudinger, \emph{Elliptic partial differential equations of second order.} Reprint of the 1998 edition. Classics in Mathematics. Springer-Verlag, Berlin, 2001.


\bibitem{IM1989}
T. Iwaniec and J. J. Manfredi, \emph{Regularity of p-harmonic functions on the plane.} Rev. Mat. Iberoamericana 5 (1989), no. 1-2, 1-19.


\bibitem{JJ2012}
V. Julin and P. Juutinen, 
\emph{A new proof for the equivalence of weak and viscosity solutions for the $p$-Laplace equation. }
Comm. Partial Differential Equations 37 (2012), no. 5, 934-946.

\bibitem{JLM2001}
P. Juutinen, P. Lindqvist and J. J. Manfredi,
\emph{On the equivalence of viscosity solutions and weak solutions for a quasi-linear equation.} 
SIAM J. Math. Anal. 33 (2001), no. 3, 699-717.



\bibitem{JLP2010}
P. Juutinen, T. Lukkari and M. Parviainen, \emph{Equivalence of viscosity and weak solutions for the $p(x)$-Laplacian.} Ann. Inst. H. Poincar\'e C Anal. Non Lin\'eaire 27 (2010), no. 6, 1471-1487.

\bibitem{kzz19}
H. Koch, Y. R. Zhang and Y. Zhou, \emph{An asymptotic sharp Sobolev regularity for planar infinity harmonic functions.} J. Math. Pures Appl. (9) 132 (2019), 457-482.


\bibitem{Lewis1983}
J. L. Lewis, \emph{Regularity of the derivatives of solutions to certain degenerate elliptic equations.} Indiana Univ. Math. J. 32 (1983), no. 6, 849-858.

\bibitem{MW1988-0}
J. J. Manfredi, \emph{$p$-Harmonic functions in the plane.} Proceedings Of The American Mathmatical Society, no. 103, 473-479.

\bibitem{MPS2000}
A. Maugeri, D. K. Palagachev and L. G. Softova, \emph{Elliptic and parabolic equations with discontinuous coefficients.}
Mathematical Research, 109. Wiley-VCH Verlag Berlin GmbH, Berlin, 2000. 256 pp.

\bibitem{MW1988}
J. J. Manfredi and A. Weitsman, \emph{On the Fatou theorem for p-harmonic functions.} Comm. Partial Differential Equations 13 (1988), no. 6, 651-668.

\bibitem{PS2008}
Y. Peres and S. Sheffield, 
\emph{Tug-of-war with noise: a game-theoretic view of the p-Laplacian. }
Duke Math. J. 145 (2008), no. 1, 91-120.

\bibitem{PS2009}
Y. Peres, O. Schramm, S. Sheffield and D. B. Wilson, \emph{Tug-of-war and the infinity Laplacian.} J. Amer. Math. Soc. 22 (2009), no. 1, 167-210.


\bibitem{RE2016}
E. Ruosteenoja,
\emph{Local regularity results for value functions of tug-of-war with noise and running payoff.} 
Adv. Calc. Var. 9 (2016), no. 1, 1-17.


\bibitem{ss2022}
S. Sarsa, \emph{Note on an elementary inequality and its application to the regularity of p-harmonic functions. }
Ann. Fenn. Math. 47 (2022), no. 1, 139-153.
{\color{red}
\bibitem{SR1993}
R. Seppo, \emph{Quasiregular Mappings. }Springer-Verlag, Berlin, 1993.
}
\bibitem{SJ2018}
J. Siltakoski, 
\emph{Equivalence of viscosity and weak solutions for the normalized $p(x)$-Laplacian. }
Calc. Var. Partial Differential Equations 57 (2018), no. 4, Paper No. 95, 20 pp.




\bibitem{SJ2022}
J. Siltakoski, 
\emph{H\"older gradient regularity for the inhomogeneous normalized $p(x)$-Laplace
equation. }
J. Math. Anal. Appl. 513 (2022), no. 1, Paper No. 126187, 27 pp.



\bibitem{Uhlenbeck1977}
K. Uhlenbeck, \emph{Regularity for a class of non-linear elliptic systems.} Acta Math. 138 (1977), no. 3-4, 219-240.

\bibitem{U1968}
N. N. Ural'ceva, \emph{ Degenerate quasilinear elliptic systems. } Zap. Naucn. Sem. Leningrad. Otdel. Mat. Inst. Steklov. (LOMI) 7(1968), 184-222.

\bibitem{W2022}
Y. Wang, C. Zhang and Y. Zhu, \emph{ A second-order Sobolev regularity for
$p(x)$-Laplace equations.} J. Math. Anal. Appl., received.

\bibitem{ZZ2012}
C. Zhang and S. Zhou, 
\emph{H\"older regularity for the gradients of solutions of the strong $p(x)$-Laplacian. }
J. Math. Anal. Appl. 389 (2012), no. 2, 1066-1077.

\bibitem{ZZ2017}
C. Zhang, X. Zhang and S. Zhou,  \emph{Gradient estimates for the strong $p(x)$-Laplace equation.} Discrete Contin. Dyn. Syst. 37 (2017), no. 7, 4109-4129.

\end{thebibliography}
\end{document}